\documentclass[reqno, 11pt]{amsart}
\usepackage{graphicx, enumerate, enumitem, amsmath, amssymb,amsthm,mathabx,manfnt,hyperref,braket,amscd,
mathrsfs,diagrams}
\usepackage[utf8]{inputenc}

\newtheorem{thm}{Theorem}
\newtheorem{prop}{Proposition}[section]
\newtheorem{lem}[prop]{Lemma}
\newtheorem{cor}[prop]{Corollary}

\theoremstyle{remark}

\numberwithin{equation}{section}

\newcommand{\Uu}{\mathfrak{U}}
\newcommand{\img}{\operatorname{img}}
\newcommand{\norm}[1]{\left\Vert#1\right\Vert}

\newcommand{\cx}{{\mathbb{C}}}
\newcommand{\rl}{{\mathbb{R}}}
\newcommand{\tensor}{\otimes}
\newcommand{\csor}{ \widehat{\otimes} }
\newcommand{\wb}{\mathbb{ W}}

\newcommand{\oc}{\mathcal O}

\newcommand{\cb}{\mathbb C}

\newcommand{\bs}{\mathcal{O}_{L^2}}
\newcommand{\ol}{\overline}

\newcommand{\dbar}{\ol\partial}
\newcommand{\wt}{\widetilde}

\title[Cohomology of annuli]{On the $L^2$-Dolbeault cohomology of annuli}

\author{Debraj Chakrabarti}
\address{Department of Mathematics, Central Michigan University, Mt. Pleasant,  MI 48859,  USA}
\email{chakr2d@cmich.edu}
\author{Christine Laurent-Thiébaut}
\address{ Institut Fourier, Universit{\'e} Grenoble Alpes,
38058 Grenoble cedex 9, France}
\email{christine.laurent@univ-grenoble-alpes.fr}
\author{Mei-Chi Shaw}
\address{Department of Mathematics, University of Notre Dame,  Notre Dame, IN 46556,  USA }
\email{mei-chi.shaw.1@nd.edu}
\thanks{\noindent Debraj Chakrabarti was partially supported by a grant from the Simons Foundation (\# 316632), by an 
Early Career internal grant from Central Michigan University, and also by a grant from the NSF (\#1600371).
  Mei-Chi Shaw was partially supported by a grant from the NSF. All three authors were partially supported by a grant from
  the  AGIR program of Grenoble INP  and Université Joseph Fourier, awarded to Christine Laurent-Thiébaut.}
\begin{document}
\begin{abstract}For certain annuli in $\cx^n$, $n\geq 2$, with non-smooth holes, we show that the $\dbar$-operator from $L^2$ functions to 
$L^2$ $(0,1)$-forms has closed range. The holes admitted include products of pseudoconvex domains and certain intersections of smoothly bounded pseudoconvex domains. 
As a consequence, we obtain estimates in 
the Sobolev space $W^1$ for the $\dbar$-equation on the non-smooth domains which are the holes of these annuli.
\end{abstract}
\maketitle
 \section{Introduction}

\subsection{Results}Let $n\geq 2$,  let $\wt \Omega$ be a bounded  domain in $\Bbb C^n$,  and $K \subset \wt\Omega$ be a non-empty compact  subset such that $\Omega=\wt \Omega \setminus K$ is connected.
 We will refer to $\Omega$ as the {\em annulus between ${K}$ and $\wt{\Omega}$}, and  $K$ as the {\em hole} of the annulus.  Annuli such as $\Omega$ are of great importance in complex analysis, for example as one of the simplest types of domains to exhibit Hartogs' phenomenon.

The goal of this paper is to study
 whether the $\dbar$-operator from $L^2_{0,0}(\Omega)$ to $L^2_{0,1}(\Omega)$ has closed range, and to characterize the range. As a consequence of such closed range results, using a duality argument, we can prove estimates for the $\dbar$-problem on the hole in degree $(0,n-1)$ in the Sobolev space $W^1$. Since our holes can be non-smooth, this leads to 
 Sobolev estimates on certain classes of non-smooth domains, including intersections of smoothly bounded convex domains. 
 Previously  such $W^1$-estimates have been known only on domains of class $\mathcal{C}^2$, and have been unknown even for domains such as the bidisc or the intersection of two balls in $\cx^2$, on both of which we obtain here $W^1$ estimates for the $\dbar$-problem.

     To state our results we introduce a few definitions. For an open subset $U\subset\cx^n$, we define  the $L^2$-Dolbeault cohomology group
\begin{equation}\label{eq-l2coho} H^{p,q}_{L^2}(U) = \frac{\ker(\dbar:L^2_{p,q}(U)\dashrightarrow L^2_{p,q+1}(U))}{\img(\dbar:L^2_{p,q-1}(U)\dashrightarrow L^2_{p,q}(U))},\end{equation}
where the dashed arrows are a reminder that the $\dbar$ operator is defined only on a dense linear subspace of the space $L^2_{p,q}(U)$. Then the quotient topology on  $H^{p,q}_{L^2}(U)$ is Hausdorff if and only if $\dbar: L^2_{p,q-1}(U)\dashrightarrow L^2_{p,q}(U)$  has closed range.  Similarly, we use the notation $H^{p,q}_{W^1}(U)$ when we substitute $L^2$ spaces by $W^1$ spaces, where $W^1(U)$ is the Sobolev space of functions in $L^2(U)$ with all first partial derivatives in
$L^2(U)$. We first note the  following strengthening of \cite[Proposition~4.7]{LS}, which characterizes the annuli with Lipschitz holes on which $\dbar$ has closed range (see also \cite[Teorema~3]{Tr1}):
\begin{thm} \label{thm-hole}
Let $V\Subset\wt\Omega$ be  bounded open subsets of $\cx^n$, $n\geq 2$. Assume $V$ has Lipschitz  boundary and $\Omega=\wt \Omega \setminus \ol V$ is connected,
 then the following are equivalent:
 \begin{enumerate}
   \item  $H_{L^2}^{0,1}(\Omega)$ is Hausdorff.
 \item   $H^{0, n-1}_{W^1}(V)=0$ and $H^{0,1}_{L^2}(\wt{\Omega})$ is Hausdorff.
 \end{enumerate}
\end{thm}


Thanks to Theorem~\ref{thm-hole}, the question of closed range in the $L^2$-sense on annuli is reduced to  an estimate in the $W^1$ norm for the $\dbar$-problem on the hole, and an $L^2$-estimate for the $\dbar$-problem on the  domain $\wt{\Omega}$. From Kohn's  theory  of the weighted $\dbar$-Neumann problem (see \cite{Ko}), it follows that for a $\mathcal{C}^\infty$-smooth bounded pseudoconvex domain $V$ in $\cx^n$, we have $H^{p,q}_{W^1}(V)=0$, if $q\geq 0$. This therefore gives 
examples of domains  to which Theorem~\ref{thm-hole} applies. 
 In this paper we give examples of more general non-smooth holes for which closed range of $\dbar$ holds in the annulus. Our first result in this direction is the following:

\begin{thm}\label{thm-new}Let $\wt{\Omega}$ be a domain in $\cx^n$, $n\geq 2$, and let $K\subset \wt{\Omega}$ be a  compact set such that $\Omega=\wt{\Omega}\setminus K$ is connected.
Suppose that
\begin{enumerate}
\item $H^{0,1}_{L^2}(\wt{\Omega})$ is Hausdorff.
\item $K=\bigcap_{j=1}^N K_j$, where for $1\leq j \leq N$, $K_j\subset \wt{\Omega}$ is a  compact set such that
$\wt{\Omega}\setminus K_j$ is connected, and  $H^{0,1}_{L^2}(\wt{\Omega}\setminus K_j)$ is Hausdorff.
 \item for each pair of indices $1\leq i,j\leq N$, the set 
 $\wt{\Omega}\setminus (K_i\cup K_j)$ is connected.
\end{enumerate}

Then $H^{0,1}_{L^2}(\Omega)$ is Hausdorff.
\end{thm}
\noindent By Theorem~\ref{thm-hole}, if $K_j$ is the closure of a Lipschitz domain $V_j$ such that $H^{0,n-1}_{W^1}(V_j)=0$,
then hypothesis (2) is satisfied (in particular, we can take $K_j=\ol{V_j}$, where $V_j$ is a smoothly bounded pseudoconvex domain). Further,  if the sets $K_j$ are taken to be closures of smoothly bounded {\em convex} domains or 
closures of smoothly bounded pseudoconvex domains which are 
{\em star-shaped} with respect to a common point, then the hypothesis (3) will be automatically satisfied. 

Our approach to Theorem~\ref{thm-new} (as well as Theorem~\ref{thm-W1} below) is based on an analog of the Leray theorem in the $L^2$ setting  which  allows us to replace questions about the $L^2$-Dolbeault cohomology $H^{0,1}_{L^2}(\Omega)$  with questions about the \v{C}ech cohomology $\check{H}^1(\Uu, \bs)$ (with coefficients in the presheaf $\bs$ of $L^2$ holomorphic functions), where  $\Uu$ is a cover  of the domain $\Omega$ by sets on each of which there is an $L^2$ estimate for the $\dbar$-operator  (see Theorem~\ref{thm-leray} below).

Combined with Theorem~\ref{thm-hole}, Theorem~\ref{thm-new} gives the following estimate for the $\dbar$-problem 
on a class of non-smooth domains:
\begin{cor}\label{cor-new} Let $V\Subset \cx^n$ be a Lipschitz domain  such that $\cx^n\setminus V$ is connected. Suppose that  $\ol{V}=\bigcap_{j=1}^N \ol{V_j}$, where for $1\leq j \leq N$, $V_j\Subset \cx^n$ is a Lipschitz domain such that $H^{0,n-1}_{W^1}(V_j)=0$. If $\cx^n\setminus V_j$ is connected for each $j$, and $\cx^n\setminus (V_i\cup V_j)$ is connected for each $1\leq i,j \leq N$,  then
$H^{0,n-1}_{W^1}(V)=0$.
\end{cor} 
If we define the Sobolev spaces correctly (see Corollary~\ref{cor-newbis} below), one can obtain analogous results for much more general domains.

In the case of intersection of {\em two} smoothly bounded pseudoconvex domains in $\cx^n$,
it is known that the $\dbar$-Neumann operator is compact in degree  $n-1$, provided it is compact on each domain (cf. \cite{ayyurustraube}). It would be interesting to know if this result is related to the above result.

 Next, we consider the  case when the hole is a {\em product. }
 
  \begin{thm}\label{thm-W1}
 Let $N\geq 2$, and for $j=1, \dots, N$, let $V_j$ be a bounded Lipschitz domain in $\cx^{n_j}$, $n_j\geq 1$, such that $\cx^{n_j}\setminus V_j$ is connected. If the dimension $n_j\geq 2$, assume further that $H^{0,n_j-1}_{W^1}(V_j)=0$.  
 
 Set $n=\sum_{j=1}^N n_j$ and let $V=V_1\times \dots\times V_N\Subset \cx^n$. Then  $H^{0,n-1}_{W^1}(V)=0$.
  \end{thm}
  When the factors are one-dimensional, a similar result holds with much less boundary regularity requirement on the factors.
  
Before Theorem~\ref{thm-W1}, it was known by a different method   (cf. \cite{prod}) that  
given $g\in W^N_{0,n-1}(V)$ such that $\dbar g=0$,
there is a $u\in W^1_{0,n-2}(V)$ such that $\dbar u=g$. Theorem~\ref{thm-W1} improves this result considerably.

Combining Theorems~\ref{thm-W1} and \ref{thm-hole}, we have the following:
\begin{cor}\label{cor-W1} Let $V\Subset \cx^n$ be a  domain which is a product as in Theorem~\ref{thm-W1}.  Let $\wt{\Omega}$ be a domain such that $H^{0,1}_{L^2}(\wt{\Omega})$ is Hausdorff, and $V\Subset \wt{\Omega}$. If $\Omega= \wt{\Omega}\setminus \ol{V}$ is connected, then $H^{0,1}_{L^2}(\Omega)$ is Hausdorff.
\end{cor}
Corollary \ref{cor-W1} solves the so-called {\em Chinese Coin Problem}, which is to obtain $L^2$-estimates for the $\dbar$
operator in an annulus $\wt{\Omega}\setminus \ol{V}$ in $\cx^2$, where $\wt\Omega$ is a ball,  and $V\Subset \wt{\Omega}$ is a bidisc (see \cite{LS2}).
Alternatively, the existence of such estimates also follows from  Theorem~\ref{thm-new}, since the bidisc can be represented
as the intersection of two smoothly bounded convex domains.

The question arises now of characterizing the cohomology group $H^{0,1}_{L^2}(\Omega)$. When the dimension $n=2$, it is known (see Corollary~\ref{cor-pseudoconvex} to Theorem~\ref{thm-fu} below) that the $L^2$-cohomology in degree $(0,1)$ of an annulus in $\cx^2$  is infinite dimensional, provided that the hole is Lipschitz.  Therefore, we need only to consider the case $n\geq 3$. 
When $\wt{\Omega}$ is strongly pseudoconvex and $K$ is also the closure of a  strongly pseudoconvex domain, 
the annulus $\Omega=\wt{\Omega}\setminus K$ satisfies the condition $Z(q)$ for $q\not = n-1$ (see \cite{Hor1, fk}). It follows 
 that for $n\geq 3$, we have for such annuli $H^{0,1}_{L^2}(\Omega)=0$.  It was shown in \cite{Sh1} and \cite[Theorem~2.2]{Sh2} that 
 even in the situation when $\wt{\Omega}$ is a smoothly bounded pseudoconvex domain and $K$ is 
 the closure of a smoothly bounded pseudoconvex domain, then we have $H^{0,1}_{L^2}(\Omega)=0$, 
 when the dimension $n\geq 3$. We can generalize this to the situation of Theorem~\ref{thm-new}:
 \begin{cor}\label{cor-new3}Let $\Omega=\wt{\Omega}\setminus K$ be an  annulus, which for some compact sets 
 $K_j$, $1\leq j\leq N$ satisfies the hypotheses (1), (2) and (3) of Theorem~\ref{thm-new}. Suppose further that 
 for each $j$, we have $H^{0,1}_{L^2}(\wt{\Omega}\setminus K_j)=0$. Then $H^{0,1}_{L^2}({\Omega})=0$.
 \end{cor} 
 For example, if $\wt{\Omega}$ is a bounded pseudoconvex domain in  $\cx^n$, $n\geq 3$ , and each $K_j\subset\wt{\Omega}$ is the closure of a smoothly bounded convex domain, then the assumptions of Corollary~\ref{cor-new3} are satisfied. To state such a vanishing result for the product situation will require further hypotheses. For simplicity, before stating the somewhat technical full result,  we first state  a special case, when each factor  of the product  is one dimensional:
\begin{cor}\label{1-dim}
Let $K_1,\dots,K_n$ be  compact sets in $\cx$ and $\wt\Omega\Subset\cx^n$ be a bounded domain with $H^{0,1}_{L^2}(\wt{\Omega})=0$. Set $K=K_1\times\dots\times K_n$ and assume that $K\subset \wt{\Omega}$ and that $\Omega= \wt{\Omega}\setminus K$ is not empty and connected. Then  $H^{0,1}_{L^2}(\Omega)$ is Hausdorff, and vanishes if $n\geq 3$.
\end{cor}
The general result on the vanishing of $H^{0,1}_{L^2}(\Omega)$ for product holes is as follows:
 \begin{thm}\label{thm-cohomology}Consider an annulus $\Omega=\wt{\Omega}\setminus K\subset\cx^n$, where
 $K=K_1\times\dots\times K_N$ is a product of compact sets in $\cx^{n_j}$, $n_j\geq 1$. Assume the following:
\begin{enumerate}
\item  $H^{0,1}_{L^2}(\wt{\Omega})=0$.
\item there is a neighborhood $U$ of $K$ contained in $\wt{\Omega}$ of the form $U= U_1\times \dots \times U_N\Subset \cx^n$, where for each $j$, the open set $U_j$
 is a neighborhood of $K_j$ in $\cx^{n_j}$,  satisfying $H^{0,1}_{L^2}(U_j)=0$.
 \item for each $j$,  $H^{0,1}_{L^2}(U_j\setminus K_j)$ is Hausdorff.
\end{enumerate}
  Then $H^{0,1}_{L^2}(\Omega)$ is Hausdorff, and  if $n\geq 3$, we have
 $H^{0,1}_{L^2}(\Omega)=0$.
\end{thm}
Note that  hypotheses (2) and (3) are vacuous if each  $n_j=1$, so that Theorem~\ref{thm-cohomology} reduces 
to Corollary~\ref{1-dim} if each of  the factors $K_j$ of $K$ is one-dimensional.  Note also that hypothesis (3) is implied
(thanks to Theorem~\ref{thm-hole}) by the following  statement:\\
{\em

\hspace{00.15in} (3$\,'$) if for some $1\leq j \leq N$ we have    $n_j\geq 2$,  then $K_j=\ol{V}_j$, where $V_j\subset \cx^{n_j}$ is a Lipschitz domain which  satisfies $H^{0,n_j-1}_{W^1}(V_j)=0$.
}


 \subsection{Remarks}

The analogs of Theorem~\ref{thm-hole} and Theorem~\ref{thm-new} continue to hold  for an annulus  in a Stein
 manifold rather than in $\cx^n$, and the proofs generalize easily. Theorem~\ref{thm-W1} and its proof can also be readily
 generalized to  products $V_1\times\dots\times V_N\subset M$ where  for each $j$, there is a Stein manifold $M_j$  of dimension $n_j$ such that  the factor $V_j$ is a relatively compact Lipschitz domain in $M_j$, further  satisfying $H^{0,n_j-1}_{W^1}(V_j)=0$ if $n_j\geq 2$, and $M$ is the product $M_1\times\dots\times M_N$. The other results can also be generalized to domains in Stein manifolds. We prefer to state the results in the case of domains in $\cx^n$ for clarity of exposition. 
 
 One also notes here that the choice of the $L^2$ topology for estimates on the $\dbar$-problem is a matter of 
 convenience rather than necessity. The methods of this paper are based on duality and gluing of local estimates, and can 
 be generalized to estimates in any norm, for example the $L^p$-norm. The required duality results in the $L^p$ setting may be found in \cite{laurentlp}.

 On pseudoconvex domains, the closed range property is a consequence of a priori estimates on the $\dbar$-operator (for bounded domains see \cite{Hor1}, and see \cite{mh15} for  some recent developments regarding unbounded domains).  When $\Omega$ is the annulus between two smooth  strongly pseudoconvex domains in $\Bbb C^n$,    $L^2$ theory    for $\dbar$  is obtained in \cite{Hor1, fk}  for all $(p,q)$-forms   since the boundary satisfies the Andreotti-Grauert condition $Z(q)$ for all $q\neq n-1$.  
 The closed range property for $\dbar$ when  $q=n-1$ also follows in this case (see Proposition 3.1.17 in \cite{fk}).
 When the domain $\Omega$  is the annulus between a bounded  pseudoconvex domain and a $C^2$-smooth pseudoconvex domain  in $\Bbb C^n$,     $L^2$ theory    for $\dbar$ has been  established in the   works \cite{Sh1, Hor2, Sh2}.      Duality between the $L^2$ theory on the annulus and the $W^1$ estimates for $\dbar$ in the hole has been 
 obtained in \cite{LS}.

The classical approach to regularity in the $\dbar$-problem is through the $\dbar$-Neumann problem.  It is difficult to use this method to obtain Sobolev estimates even on  simple Lipschitz domains such as the bidisc or the intersection of two balls. The problem arises because a $(0,1)$-form on the bidisc $\mathbb{D}^2$ which is in the domain of $\dbar^*$  and is smooth up to the boundary must vanish along the \v{S}ilov boundary $b\mathbb{D}\times b\mathbb{D}$, since the complex normal components of the form along two $\cx$-linearly independent directions must vanish. Consequently, 
a priori estimates  do not translate into actual estimates. In fact, one can show that on product domains, the $\dbar$-Neumann operator does not preserve the space of $(0,1)$-forms smooth up to the boundary on the bidisc (see \cite{ehsani-bidisc}).  However, for domains represented as intersections of strongly pseudoconvex domains, with boundaries meeting transversely, one may obtain subelliptic estimates with $\frac{1}{2}$ gain for the canonical solution operator (see \cite{ms1998,hik}). 

\medskip

\subsection{Acknowledgements} We gratefully acknowledge the valuable comments and suggestions of Y.-T. Siu on the topic of this paper. We are also thankful to the referee for helpful suggestions. 
\section{General condition for closed range}
\subsection{Notation and preliminaries} \label{sec-notation} We introduce some notations for the $L^2$-version of the $\dbar$-complex.

For an open set $U\subset \cx^n$, denote
\begin{equation}\label{eq-apq}
 A^{p,q}_{L^2}(U)= \{f\in L^2_{p,q}(U)| \dbar f \in L^2_{p,q+1}(U)\},
\end{equation}
where  $\dbar$ acts in the sense of distributions. Recall that this defines the (weak) maximal realization of the $\dbar$-operator as an unbounded densely defined closed operator on $L^2_{p,q}(U)$.
 Let $Z^{p,q}_{L^2}(U)=\{f\in A^{p,q}_{L^2}(U)|\dbar f =0\}$ and $B^{p,q}_{L^2}(U)=\{\dbar f|f\in A^{p,q-1}_{L^2}(U)\} $ denote
the subspaces of $\dbar$-closed and $\dbar$-exact forms in $A^{p,q}_{L^2}(U)$. Note that $B^{p,q}_{L^2}(U)\subset
Z^{p,q}_{L^2}(U)$ since $\dbar^2=0$. As noted in \eqref{eq-l2coho}, we denote $H^{p,q}_{L^2}(U)= Z^{p,q}_{L^2}(U)/B^{p,q}_{L^2}(U)$.

Denote by $\dbar_c$ the (strong) minimal realization of the $\dbar$-operator as a closed, densely defined, unbounded linear
operator from $L^2_{p,q}(U)$ to $L^2_{p,q+1}(U)$.  The operator $\dbar_c$ is the closure (in graph norm) of the restriction of $\dbar$ to the space of smooth compactly supported forms in $L^2_{p,q}(U)$. The domain of the
operator $\dbar_c$ is a dense subspace of $L^2_{p,q}(U)$,  denoted by $A^{p,q}_{c,L^2}(U)$. Then $A^{p,q}_{c,L^2}(U)$ is a proper subspace of $A^{p,q}_{L^2}(U)$ for bounded domains $U$, and  $\dbar_c$ is the restriction to
 $A^{p,q}_{c,L^2}(U)$ of the operator $\dbar :A^{p,q}_{L^2}(U)\to B^{p,q+1}_{L^2}(U)$.
We denote by
$Z^{p,q}_{c,L^2}(U)=\{f \in A^{p,q}_{c,L^2}(U)| \dbar_c f=0\}$ the space of $\dbar_c$-closed forms,
and by $B^{p,q}_{c, L^2}(U)=\{\dbar_c f| f \in A^{p,q}_{c,L^2}(U)\}$ the space of $\dbar_c$-exact forms. The quotient $H^{p,q}_{c,L^2}(U)= Z^{p,q}_{c,L^2}(U)/B^{p,q}_{c, L^2}(U)$ is the $L^2$-Dolbeault cohomology with minimal realization, analogous to cohomology with compact support. We will also  use the following $L^2$-analog of Serre duality. Consider the pairing $H^{p,q}_{L^2}(U)\times H^{n-p, n-q}_{c, L^2}(U)\to \cx$
given by
\[ ({\rm class}(f), {\rm class}(g))\mapsto \int_U f\wedge g.\]
Then $H^{p,q}_{L^2}(U)$ (resp.   $H^{n-p, n-q}_{c, L^2}(U)$) is Hausdorff if and only if 0 is the only element of $H^{p,q}_{L^2}(U)$ (resp.  $H^{n-p, n-q}_{c, L^2}(U)$)  which is orthogonal to all of $H^{n-p, n-q}_{c, L^2}(U)$ (resp. $H^{p,q}_{L^2}(U)$) (cf. \cite{ChaSh}).

We say that a  nonempty compact subset  $K\subset\rl^n$ is {\em regular} if there is an open set $V$ with
\begin{equation}\label{eq-regularity}
K= \ol{V} \text{\quad and \quad}V= \mathop{\rm interior}(K).
\end{equation}
For  regular compact subset $K\subset \cx^n$, we will use the notation
\begin{equation}\label{eq-hw1}
  H^{p,q}_{W^1}(K)=0\end{equation}
to denote that the following is true: if $f\in W^1_{p,q}(\cx^n)$ is a form with coefficients in the Sobolev space 
$W^1$, and on the set $K$ we have $\dbar f=0$, then there is a form $u\in W^1_{p,q-1}(\cx^n)$ such that on the set $K$ we have $\dbar u=f$. Note that if $K=\ol{V}$ where $V$ is a Lipschitz domain, then $H^{p,q}_{W^1}(K)=0$ if and only if $H^{p,q}_{W^1}(V)=0$. This is so since each function in $W^1(V)$ can be extended to a function in $W^1(\cx^n)$.

\subsection{Proof of Theorem~\ref{thm-hole}} We prove below three propositions 
which together imply Theorem~\ref{thm-hole}:

\begin{prop}\label{prop-hole1}Let $\wt{\Omega}$ be a bounded open set in $\cx^n$, $n\geq 2$, and let $K$ be a regular compact subset of  $\wt{\Omega}$. If  $\Omega=\wt \Omega \setminus  K$ is connected, and such that
 $H_{L^2}^{0,1}(\Omega)$ is Hausdorff. Then $H^{0, n-1}_{W^1}(K)=0$.
\end{prop}
\begin{proof} Let $f\in W^1_{0,n-1}(\cx^n)$ be a form with $W^1$ coefficients, and  assume that $\dbar f=0$ on $K$. We need to show that there is a $u\in W^1_{0,n-2}(\cx^n)$ such that $\dbar u =f$ holds on $K$.
After multiplying with an appropriate smooth compactly supported function, we may assume that ${f}$ has compact support in $\wt{\Omega}$. 
Then the $(0,n)$-form $\dbar {f}$ lies in $ A^{0,n}_{c,L^2}(\wt{\Omega})$ and vanishes on $K$. We claim  that for each holomorphic $(n,0)$-form $\theta \in Z^{n,0}_{L^2}(\Omega)$  on
$\Omega$ we have
\begin{equation}\label{eq-orthogonal}
 \int_\Omega \theta\wedge \dbar {f}=0.
\end{equation}
Indeed, since $\Omega$ is connected, by the Hartogs extension phenomenon, the form $\theta$ extends through the hole $K$  to give a holomorphic
form $\wt{\theta}\in Z^{n,0}_{L^2}(\wt{\Omega})$ and
\[ \int_\Omega \theta  \wedge \dbar {f}= \int_{\wt{\Omega}}\wt{\theta}\wedge \dbar {f}=(-1)^n\int_{\wt{\Omega}}\dbar \wt{\theta}\wedge {f} =0.\]
Now by  \cite[Lemma~3]{ChaSh}, the hypothesis $H_{L^2}^{0,1}(\Omega)$ is Hausdorff is equivalent to the fact that
$H^{n,n}_{c,L^2}(\Omega)$ is Hausdorff, which in turn is equivalent to $H^{0,n}_{c,L^2}(\Omega)$ being Hausdorff.
Now \eqref{eq-orthogonal} shows that under the Serre pairing $H^{n,0}_{L^2}(\Omega)\times H^{0,n}_{c,L^2}(\Omega)\to \cx$, the cohomology class  of $\dbar f$ is orthogonal to all of $H^{n,0}_{L^2}(\Omega)$,
and therefore there is a $g\in  A^{0,n-1}_{c,L^2}(\Omega)$ such that $\dbar {f}=\dbar g$. Let $\wt{g}$ be the extension of $g$ by 0 to all of $\cx^n$. Then  $\wt{g}\in A^{0,n-1}_{c,L^2}(\cx^n)$  and has compact support. Therefore, the form ${f}-\wt{g}$ is a compactly supported $\dbar$-closed form in $Z^{0,n-1}_{c, L^2}(\cx^n)$.
Thus there is a compactly supported $(0,n-2)$-form $u$ on $\cx^n$ such that $\dbar u={f}-\wt{g}$,
and by interior regularity of the $\dbar$-problem, we may assume that $u$ has coefficients in $W^1$ when restricted to a neighborhood of $K$.
 Noting that by construction $\wt{g}$ vanishes on $K$, we see that  $\dbar u = f$, so that  $H^{0, n-1}_{W^1}(K)=0$.
 \end{proof}
 For a general regular compact set $K$, the condition $H^{0,n-1}_{W^1}(K)=0$ is only a statement about existence of 
solutions of the $\dbar$ problem, and does not lead to any estimates for these solutions. However, when $H^{0,n-1}_{W^1}(K)=0$ can be  interpreted as the vanishing of a cohomology defined by a densely-defined closed realization of the operator $\dbar$ acting between Banach spaces, by the open mapping theorem, we do obtain estimates for the $\dbar$-problem. For example, when $K=\ol{V}$, where $V$ is a Lipschitz domain, the condition $H^{0,n-1}_{W^1}(K)=0$ implies that there is a constant $C>0$ such that for each $f\in W^1_{0,n-1}(V)$ such that $\dbar f=0$ as distributions, there is a $u\in W^1_{0,n-2}(V)$ such that $\dbar u =f$ in the sense of distributions, and $\norm{u}_{W^1} \leq C \norm{f}_{W^1}$.
\begin{prop}\label{prop-hole2}Let $\wt{\Omega}$ be a bounded open set in $\cx^n$, $n\geq 2$, and let $K$ be a compact subset of  $\wt{\Omega}$. If  $\Omega=\wt \Omega \setminus  K$ is connected, and such that
 $H_{L^2}^{0,1}(\Omega)$ is Hausdorff. Then $H^{0,1}_{L^2}(\wt{\Omega})$ is Hausdorff.
\end{prop}
\begin{proof}

Assuming again that $H_{L^2}^{0,1}(\Omega)$ is Hausdorff, we now show that
$H^{0,1}_{L^2}(\wt{\Omega})$ is Hausdorff. Let $f\in Z^{0,1}_{L^2}(\wt{\Omega})$ be such that for each $\phi \in Z^{n,n-1}_{c, L^2}(\wt{\Omega})$, we have $\int_{\wt{\Omega}}f\wedge \phi=0$. We have to show that   $f$ is $\dbar$-exact, and then the claim will follow by Serre duality.

Now if $\psi\in Z^{n,n-1}_{c,L^2}(\Omega)$, then we clearly have $\int_{\Omega}f\wedge \psi=0$, so that the fact that
$H^{0,1}_{L^2}(\Omega)$ is Hausdorff implies that there is a $g\in A^{0,0}_{L^2}(\Omega)$ satisfying $\dbar g =f$. Let $\chi\in \mathcal{C}_0^\infty(\cx^n)$ be a cutoff which  is identically equal to 1  on a neighborhood of $K$ and is supported in a compact subset of $\wt{\Omega}$. We consider the $(0,1)$-form $\theta$  on $\cx^n$ defined by extending   by zero the compactly supported form  $f- \dbar\left((1-\chi)g\right)$ on $\Omega$. Writing $f- \dbar\left((1-\chi)g\right)=\chi f +\dbar\chi\wedge g$, we see that
$\theta$ is a compactly supported form in $L^2_{0,1}(\cx^n)$.  Therefore there is a compactly supported $L^2$ function $u$  on $\cx^n$ such that $\dbar u =\theta$. Then $(1-\chi) g +u \in A^{0,0}_{L^2}(\wt{\Omega})$, and $\dbar\left((1-\chi) g+u\right)= f$.
\end{proof}

\begin{prop} \label{prop-hole3}
Let $V\Subset\wt\Omega$ be  bounded open subsets of $\cx^n$, $n\geq 2$. Assume $V$ has Lipschitz  boundary and $\Omega=\wt \Omega \setminus \ol V$ is connected, $H^{0, n-1}_{W^1}(V)=0$ and $H^{0,1}_{L^2}(\wt{\Omega})$ is Hausdorff. Then  $H_{L^2}^{0,1}(\Omega)$ is Hausdorff.
 \end{prop}
 \begin{proof}
 
   First note that thanks to \cite[Lemma~3]{ChaSh}, the hypothesis that $H^{0,1}_{L^2}(\wt{\Omega})$ is Hausdorff is equivalent to the condition that $H^{0,n}_{c,L^2}(\wt{\Omega})$ is Hausdorff, and the conclusion that $H^{0,1}_{L^2}({\Omega})$ is Hausdorff is equivalent to  $H^{0,n}_{c,L^2}({\Omega})$  being Hausdorff. We therefore start with an $f\in Z^{0,n}_{c,L^2}(\Omega)$  such that for each $\phi \in Z^{n,0}_{L^2}(\Omega)$ we have $\int_\Omega f\wedge\phi=0$, and want to show that there is a $u$ in $A^{0,n-1}_{c,L^2}(\Omega)$ such that $\dbar u =f$.

By Hartogs' phenomenon, each element of $Z^{n,0}_{L^2}(\Omega)$ extends to an element of  $ Z^{n,0}_{L^2}(\wt{\Omega})$. It follows that $\int_{\wt{\Omega}} \wt{f}\wedge \phi=0$, where $\wt{f}$ is the extension of $f$ by 0 to
all of $\wt{\Omega}$. Since by hypothesis, $H^{0,n}_{c,L^2}(\wt{\Omega})$ is Hausdorff, there is a $g\in A^{0,n-1}_{c, L^2}(\wt{\Omega})$ such that $\dbar g= \wt f$, and using interior regularity of $\dbar$, we may assume that $g$ has $W^1$ coefficients in a neighborhood of $\ol{V}$. Note that by the definition of $\wt{f}$, we have $\dbar g=0$ on $V$.  Invoking the hypothesis $H^{0,n-1}_{W^1}(V)=0$,
we see that there is an $h\in W^1_{0,n-2}(V)$ such that $\dbar h =g$ on $V$. Since $V$ is Lipschitz, we may extend $h$ 
as a form with $W^1$ coefficients on the whole of $\cx^n$.  Multiplying by a smooth cutoff, we may further
assume that $h$ has compact support in $\wt{\Omega}$.
Then $u=(g-\dbar {h})|_\Omega$ is a form on $\Omega$ whose extension by zero to $\wt{\Omega}$ belongs to 
the domain of $\dbar$ on $\wt{\Omega}$. Since $V$ has Lipschitz boundary, it follows by \cite[Proposition~2]{ChaSh} that $u$ belongs to $A^{0,n-1}_{c,L^2}(\Omega)$, and $\dbar u= \dbar g = f$ on $\Omega$.  
\end{proof}


\subsection{The two-dimensional case}

 For a domain $D\subset\cx^n$, denote by $Z^{0,1}(\overline{D})$ the space
$\mathcal{C}^\infty_{0,1}(\overline{D})\cap \ker \dbar$  of $\dbar$-closed forms which are $\mathcal{C}^\infty$-smooth up to the 
boundary  on ${D}$, and let $\wt{B}^{0,1}(D)$ denote the subspace of $Z^{0,1}(\overline{D})$ consisting of those 
forms $g$ smooth up to the boundary such that there is a {\em distribution} $u\in \mathcal{D}'(D)$  such that $\dbar u =g$ on $D$.  We consider the quotient
\begin{equation}\label{eq-htilde}
\wt{H}^{0,1}(D)=Z^{0,1}(\overline{D})/\wt{B}^{0,1}(D).
\end{equation}
The proof of the following result  is known   (see Fu  \cite{Fu}):
\begin{thm}\label{thm-fu} Let $D\subset\cx^2$ be a domain such that ${\rm interior}(\ol D)=D$. If the vector space $\wt{H}^{0,1}(D)$ of \eqref{eq-htilde} is finite dimensional, then $D$ is pseudoconvex. 
\end{thm}
\begin{proof}We repeat the proof of \cite{Fu}.
Assume $D$ is not pseudoconvex, then there exists a domain $\wt D$ strictly containing $D$ such that any holomorphic function on $D$ extends holomorphically to $\wt D$. Since interior($\ol D$)$=D$, after a translation and a rotation we may assume that $0\in\wt D\setminus\ol D$ and there exists a point $z_0$ in the intersection of the plane $\{(z_1,z_2)\in\cb^2~|~z_1=0\}$ with $D$, which belongs to the same connected component  of that plane with $\wt D$.

For an integer $k\geq 0$, we consider the smooth  $(0,1)$-form $B_k$ on $\cx^2\setminus \{0\}$ derived from the Bochner-Martinelli kernel and given by 
\[ B_k(z_1,z_2) = (k+1)!\cdot\ol z_2^k\cdot\frac{\ol z_1~d\ol z_2-\ol z_2~d\ol z_1}{|z|^{2(k+2)}}.\]
These forms  are $\dbar$-closed, and if we define
\[ u_k(z_1,z_2)=k! \cdot\frac{\ol z_2^{(k+1)}}{|z|^{2(k+1)}},\]
we have on $\cx^2\setminus\{0\}$:
\[ \dbar (u_k)= -z_1B_k.\]
Note that the restriction of $B_k$ to $D$ belongs to $Z^{0,1}(\ol{D})$. Let $N$ be an integer such that $N>\dim \wt{H}^{0,1}(D)$. Then there exists a non trivial linear combination $B=\sum_{k=1}^N a_kB_k$, (which belongs to $Z^{0,1}(\ol{D})$), such that
there exists a distribution $v$ on $D$ satisfying $\dbar v=B$. 
Set 
\[F=z_1v+\sum_{k=1}^N a_k u_k,\] 
then $F$ is a holomorphic function on $D$, so it should extend holomorphically to $\wt D$ but we have \[ F(0,z_2)=\sum_{k=1}^N a_k\frac{k!}{z_2^{k+1}},\] which is holomorphic and singular at $z_2=0$, which gives the contradiction since $0\in\wt D\setminus \overline D$.

\end{proof}
This allows us to prove the following analog for cohomologies with estimates of a result of Laufer (\cite{lauf}).
\begin{cor}\label{cor-pseudoconvex}
Let $D\Subset\cb^2$ be a bounded domain such that interior($\ol D$)$=D$. If either  $H^{0,1}_{L^2}(D)$ or $H^{0,1}_{W^1}(D)$ is finite dimensional, then $D$ is pseudoconvex.
\end{cor}
\begin{proof} Suppose that $H^{0,1}_{L^2}(D)=Z^{0,1}_{L^2}(D)/B^{0,1}_{L^2}(D)$ is finite dimensional. Then a fortiori, the space
$Z^{0,1}(\ol{D})/(B^{0,1}_{L^2}(D)\cap Z^{0,1}(\ol{D}))$ consisting of $L^2$ cohomology classes representable by forms smooth up to the boundary is also finite
dimensional.  But
\[  \frac{Z^{0,1}(\ol{D})}{B^{0,1}_{L^2}(D)\cap Z^{0,1}(\ol{D})}\supset\frac{Z^{0,1}(\ol{D})}{\wt{B}^{0,1}(D)}= \wt{H}^{0,1}(D), \]
since $B^{0,1}_{L^2}(D)\cap Z^{0,1}(\ol{D})\subset \wt{B}^{0,1}(D)$.  Therefore it follows that $\wt{H}^{0,1}(D)$ is finite dimensional, and we can apply Theorem~\ref{thm-fu} to conclude that $D$ is pseudoconvex. Exactly a similar proof
works in the case of the $W^1$-cohomology.
\end{proof}
We can now give a general characterization of domains in $\cx^2$ on which the $L^2$ $\dbar$-operator has closed range:
\begin{cor}\label{cor-2dim}
Let $V \Subset \wt\Omega$ be Lipschitz domains in  $\cx^2$,   and  suppose that $\cx^2\setminus \wt{\Omega}$ and 
$\Omega=\wt \Omega \setminus V$ are connected.  If  $H_{L^2}^{0,1}(\Omega)$ is Hausdorff then both $\wt{\Omega}$ and 
$V$ are pseudoconvex, and  the space $H^{0,1}_{L^2}(\Omega)$ is  infinite dimensional.
\end{cor}
\begin{proof} By Theorem~1, the fact that $H^{0,1}_{L^2}(\Omega)$ is Hausdorff is equivalent to 
$H^{0,1}_{W^1}(V)=0$ and $H^{0,1}_{L^2}(\wt{\Omega})$ being Hausdorff. By Corollary~\ref{cor-pseudoconvex} above,  $H^{0,1}_{W^1}(V)=0$ implies that $V$ is pseudoconvex. 
 It was shown in \cite[Theorem~3.4]{LS} that for a Lipschitz ``hole-less'' domain such as $\wt{\Omega}$,  the space $H^{0,1}_{L^2}(\wt{\Omega})$ is Hausdorff  if and only if $\wt{\Omega}$ is pseudoconvex.

The infinite-dimensionality of $H^{0,1}_{L^2}(\Omega)$ also follows from Corollary~\ref{cor-pseudoconvex}, since we have ${\rm interior}(\ol{\Omega})= \Omega$, and  if $H^{0,1}_{L^2}(\Omega)$ were finite dimensional, $\Omega=\wt{\Omega}\setminus \ol{V}$ would be pseudoconvex.
\end{proof}

\section{$L^2$-Dolbeault and $L^2$-\v{C}ech cohomologies}
For an open $U\subset\cx^n$, let $A^{p,q}_{L^2}(U), B^{p,q}_{L^2}(U)$ and $Z^{p,q}_{L^2}(U)$ be as in section~\ref{sec-notation} above.
Then $A^{p,q}_{L^2}$ defines a presheaf of pre-Hilbert spaces on $\Omega$, which  is not a sheaf.  $ B^{p,q}_{L^2}$ and $Z^{p,q}_{L^2}$  are sub-presheaves of $A^{p,q}_{L^2}$, and    $Z^{p,q}_{L^2}$ is a presheaf of Hilbert spaces.  Note that $Z^{0,0}_{L^2}(U)$ is just the
Bergman space on $U$, which will also be denoted  by $\mathcal{O}_{L^2}(U)$. Recall also that the $L^2$-Dolbeault cohomology of $U$
is defined to be the quotient topological vector space $H^{p,q}_{L^2}(U)= Z^{p,q}_{L^2}(U)/B^{p,q}_{L^2}(U)$.
Given open sets $V\subset U\subset \Omega$, and a class $\gamma\in H^{p,q}_{L^2}(U)$, we denote by $\gamma|_{V}$ the class in $H^{p,q}_{L^2}(V)$ obtained by restricting $\gamma$ to $V$. More precisely, if $g\in Z^{p,q}_{L^2}(U)$ is such that $\gamma={\rm class}(g)$ in $H^{p,q}_{L^2}(U)$, then
$\gamma|_V= {\rm class}(g|_V)$ in  $H^{0,1}_{L^2}(V)$.

Now suppose that we are given a finite collection  $\Uu=\{\Omega_j\}_{j=1}^N$ of open sets
covering an open set  $\Omega\subset \cx^n$ (i.e., $\bigcup_{j=1}^N \Omega_j=\Omega$).    Given a presheaf $\mathcal{B}$ of normed linear spaces on $\Omega$ (e.g. the presheaves $A^{p,q}_{L^2}, Z^{p,q}_{L^2}$ and $B^{p,q}_{L^2}$ of the previous paragraph), we can define for each $k$, the space  $\check{C}^k(\Uu,\mathcal{B})$ of \v{C}ech $k$-cochains, and the corresponding coboundary map $\delta_k:\check{C}^k(\mathfrak{U}, \mathcal{B})\to \check{C}^{k+1}(\mathfrak{U}, \mathcal{B})$  (cf. \cite[p. 187]{GuRo}). As usual, we let $\check{Z}^k(\Uu,\mathcal{B})$ and $\check{B}^k(\Uu,\mathcal{B})$ denote the spaces of \v{C}ech cocycles and coboundaries respectively, and then the
\v{C}ech cohomology of $\mathcal{B}$ with respect to the cover $\Uu$ is given by $\check{H}^q(\Uu, \mathcal{B})=\check{Z}^k(\Uu,\mathcal{B})/\check{B}^k(\Uu,\mathcal{B})$.

  $\check{C}^k(\mathfrak{U}, \mathcal{B})$ is a topological vector space as the direct sum of the $\mathcal{B}(\Omega_{i_0\dots i_k})$ as
$1\leq i_0,\dots,i_k\leq N$,  endowed with the direct sum topology, where $\Omega_{i_0\dots i_k} =\bigcap_{\ell=0}^k \Omega_{i_\ell}$.  The topological vector space  $\check{C}^k(\mathfrak{U}, \mathcal{B})$ is a normed linear space in a natural way:  for $k\geq 0$, a norm on
$\check{C}^k(\mathfrak{U}, \mathcal{B})$ is given by
\[ \norm{g}_{\check{C}^k(\mathfrak{U}, \mathcal{B})}^2 = \sum_{1\leq i_0,\dots,i_k\leq N}\norm{g_{i_0\dots i_k}}_{\mathcal{B}(\Omega_{i_0\dots i_k})}^2.\]
  Of course  there are many other choices of equivalent norms, but the above choice is appropriate when $\mathcal{B}$ is a sheaf of pre-Hilbert spaces, which is the only case we consider.  Then $\check{C}^k(\mathfrak{U}, \mathcal{B})$ is again a pre-Hilbert space, and a Hilbert space if $\mathcal{B}$ happens to be a sheaf of Hilbert spaces.
With this topology, the coboundary map
$\delta$ is continuous, the cocycle space $\check{Z}^k(\mathfrak{U}, \mathcal{B})= \ker \delta_k\cap \check{C}^k(\mathfrak{U}, \mathcal{B})$ is  a normed linear space, and the
 coboundary space $ \check{B}^k(\mathfrak{U}, \mathcal{B})= \img \delta_{k-1}\cap \check{C}^k(\mathfrak{U}, \mathcal{B})$ is a subspace of
 $\check{Z}^k(\mathfrak{U}, \mathcal{B})$.  Then the $k$-th cohomology group of this complex $\check{H}^k(\Uu,\mathcal{B})$ is a topological
 vector space with the quotient topology, and this topology is Hausdorff if and only if $\check{B}^k(\mathfrak{U}, \mathcal{B})$ is
 a closed subspace of  $\check{Z}^k(\mathfrak{U}, \mathcal{B})$.

 A relation between these two types of cohomology is given by the following result, whose proof is inspired by that of a well-known classical result of Leray (cf. \cite[page 189]{GuRo}). Related results were obtained for
  the Fréchet topology in \cite{laufer}.

\begin{thm}\label{thm-leray} Let $\Omega$ be a bounded domain in $\cx^n$.  Let $\Uu=\{\Omega_1,\dots,\Omega_N\}$ be a finite open cover of $\Omega$  such that, for all $j=1,\dots,N$, the cohomology group $H^{0,1}_{L^2}(\Omega_j)$ is Hausdorff. Then,

 \begin{enumerate}
 \item
   if  the \v{C}ech group $\check{H}^1(\Uu,Z^{0,0}_{L^2})$ is Hausdorff, then the
$L^2$-Dolbeault cohomology $H^{0,1}_{L^2}(\Omega)$ is also Hausdorff.
\item  if $\check{H}^1(\Uu,Z^{0,0}_{L^2})=0$,  the  map  $H^{0,1}_{L^2}(\Omega)\to \bigoplus_{k=1}^N H^{0,1}_{L^2}(\Omega_k) $, induced by restriction maps from $\Omega$ to $\Omega_k$, maps $H^{0,1}_{L^2}(\Omega)$ isomorphically onto a subspace of
\begin{equation}\label{eq-image2}
\left\{ \left(\gamma_k\right)_{k=1}^N, \gamma_k \in H^{0,1}_{L^2}(\Omega_k) \left| \text{  for $i\not=j$}, \gamma_i|_{\Omega_{ij}}= \gamma_j|_{\Omega_{ij}}\right.\right\}.
\end{equation}
\item If  $\check{H}^1(\Uu,Z^{0,0}_{L^2})=0$ and for each $j$,  $H^{0,1}_{L^2}(\Omega_j)=0$, then $H^{0,1}_{L^2}(\Omega)=0$.
\end{enumerate}\end{thm}
This result may be interpreted as saying that given an $L^2$-estimate for the $\dbar$-problem in each open set of the cover $\Uu$, the obstruction to obtaining a global $L^2$-estimate on $\Omega$ resides in the $L^2$-\v{C}ech group $\check{H}^1(\Uu,Z^{0,0}_{L^2})$.  When this \v{C}ech group is Hausdorff, then the $L^2$-Dolbeault group is Hausdorff. In particular, when the \v{C}ech group vanishes, the  cohomology classes in $H^{0,1}_{L^2}(\Omega)$ are obtained by ``gluing together" the cohomologies in each set in the cover as in \eqref{eq-image2}.

We also remark that use of the $L^2$-topology in Theorem~\ref{thm-leray} is not important, and similar gluing 
techniques work for estimates in any norm, e.g., $L^p$ estimates or Hölder estimates. 

\begin{proof}[Proof of Theorem~\ref{thm-leray}] For $k,q\geq 0$, define an operator \[\dbar=\dbar_q: \check{C}^k(\Uu, A^{p,q}_{L^2})\to \check{C}^{k}(\Uu, A^{p,q+1}_{L^2})\] sectionwise, i.e., for $g\in \check{C}^k(\Uu, A^{p,q}_{L^2})$, set
$ (\dbar g)_{i_0\dots i_k} = \dbar (g_{i_0\dots i_k}),$
and note that for a given $k$, $\ker \dbar_q=\check{C}^k(\Uu,Z^{p,q}_{L^2})$ and $\img \dbar_q =\check{C}^k(\Uu, B^{p,q+1}_{L^2})$.
Then for each fixed $k,p\geq 0$  we have a complex $\left(\check{C}^k(\Uu,A^{p,q}_{L^2}),\dbar_q\right)$. In fact, we have,
for each fixed $p$ a double complex of commuting differentials $\left(\check{C}^k(\Uu,A^{p,q}_{L^2}),\dbar_q, \delta_k\right)$: i.e. we have  $\dbar{} \delta = \delta {} \dbar$. This  follows since for $g\in \check{C}^k(\Uu, A_{p,q})$,
\[ \dbar \delta g = \delta \dbar g = \sum_{j=0}^k (-1)^j\dbar g_{i_0\dots \widehat{i_j} \dots i_k}\]
where the hat denotes omission. We represent the relevant part of the double complex for $p=0$ in the following diagram:
\[\begin{diagram}
0&\rTo &Z^{0,0}_{L^2}(\Omega) &\rTo^\varepsilon& \check{C}^0(\Uu, Z^{0,0}_{L^2})&\rTo^\delta &\check{C}^1(\Uu,Z^{0,0}_{L^2})&\rTo^{\delta}&\dots\\
&&\dTo^ i&&\dTo^i&&\dTo^i\\
0&\rTo &A^{0,0}_{L^2}(\Omega) &\rTo^\varepsilon& \check{C}^0(\Uu, A^{0,0}_{L^2})&\rTo^\delta &\check{C}^1(\Uu,A^{0,0}_{L^2})&\rTo^{\delta}&\dots\\
&&\dTo^\dbar &&\dTo^\dbar&&\dTo^\dbar\\
0&\rTo &A^{0,1}_{L^2}(\Omega) &\rTo^\varepsilon& \check{C}^0(\Uu, A^{0,1}_{L^2})&\rTo^\delta &\check{C}^1(\Uu,A^{0,1}_{L^2})&\rTo^{\delta}&\dots\\
&&\dTo^\dbar &&\dTo^\dbar&&\dTo^\dbar\\
\end{diagram}\]
Here, for a presheaf $\mathcal{B}$, the map    $\varepsilon: \mathcal{B}(\Omega)\to \check{C}^0(\mathfrak{U},\mathcal{B})$ is given  by  $ (\varepsilon f)_i = f|_{\Omega_i}$
for each $\Omega_i \in \mathfrak{U}$. Note that then the sequence $\mathcal{B}(\Omega)\xrightarrow{\varepsilon}\check{C}^0(\mathfrak{U},\mathcal{B})\xrightarrow{\delta}\check{C}^1(\mathfrak{U},\mathcal{B})$ is exact at $\check{C}^0(\mathfrak{U},\mathcal{B})$, i.e.,  $\img \varepsilon = \ker \delta$. The map $i$ is the inclusion map,
so that we have that each vertical column is exact along the second row, i.e., $\img i = \ker \dbar$. Our proof will be a ``Topological Diagram Chase'' with this diagram,
where we will need to keep track of the continuity of the maps.

By hypothesis  $H^{0,1}_{L^2}(\Omega_j)$ is Hausdorff for each $j$,  so  there is a continuous solution operator $K_j:B^{0,1}_{L^2}(\Omega_j)\to A^{0,0}_{L^2}( \Omega_j)$. For example, we can take $K_j$ to be the canonical solution operator $\dbar^* N_{0,1}$, where $\dbar^*$ is the adjoint of $\dbar$, and $N_{0,1}$ is the $\dbar$-Neumann operator (see \cite{fk}).
We define a map $K=K:\check{C}^0(\mathfrak{U},Z^{0,1}_{L^2})\to \check{C}^0(\mathfrak{U}, A^{0,0}_{L^2})$, by setting $(Kg)_j=K_j(g_j)$.
Then $K$ is continuous, and we have $\dbar K g =g$.

Consider the map $\varepsilon: A^{0,1}_{L^2}(\Omega)\to \check{C}^0(\Uu, A^{0,1}_{L^2})$. Since by hypothesis, $B^{0,1}_{L^2}(\Omega_j)$ is a closed subspace of $A^{0,1}_{L^2}(\Omega_j)$ for each $j$, it follows that
$\varepsilon^{-1}\left(\check{C}^0(\Uu,B^{0,1}_{L^2})\right)$ is a closed subspace of $ A^{0,1}_{L^2}(\Omega)$.  Define a continuous linear map
\[ \ell: \varepsilon^{-1}\left(\check{C}^0(\Uu,B^{0,1}_{L^2})\right)\to \check{Z}^1(\Uu,Z^{0,0}_{L^2})\] by setting
\begin{equation}\label{eq-ell}
\ell= \delta{} K {} \varepsilon.
\end{equation}
From the definitions of $\varepsilon, K$ and $\delta$,  this is a continuous linear map from  $\varepsilon^{-1}\left(\check{C}^0(\Uu,B^{0,1}_{L^2})\right)$ to $\check{C}^1(\Uu,A^{0,0}_{L^2})$. But for $g\in \varepsilon^{-1}\left(\check{C}^0(\Uu,B^{0,1}_{L^2})\right)$, we have
\[
\dbar\ell g= \dbar {}\delta {} K {} \varepsilon g
= \delta {} \dbar {} K {} \varepsilon g
= \delta {} \varepsilon g
=0,\]
which shows that $\ell(g)\in \check{C}^1(\Uu,Z^{0,0}_{L^2})$.  However, since $\delta{}\delta=0$, it follows that
$\ell(g)\in \check{Z}^1(\Uu,Z^{0,0}_{L^2})$.

The basic property of $\ell$ that makes it useful is that
\begin{equation}\label{eq-ellproperty}
\ell^{-1}\left(\check{B}^1\left(\Uu,Z^{0,0}_{L^2}\right)\right) = B^{0,1}_{L^2}(\Omega).
\end{equation}

To see \eqref{eq-ellproperty},  first, let $g\in B^{0,1}_{L^2}(\Omega)$. Then there is a $u\in A^{0,0}_{L^2}(\Omega)$ such that $\dbar u =g$. Consider
\[ h=K{} \varepsilon g - \varepsilon u.\]
By construction, $h\in \check{C}^0(\Uu, A^{0,0}_{L^2})$.  Note  that
\[ \dbar h =  \dbar {} K {} \varepsilon g - \dbar {}\varepsilon u= \varepsilon g - \dbar \varepsilon u =0,\]
since $\dbar u=g$. Therefore, in fact, $h\in \check{C}^0(\Uu, Z^{0,0}_{L^2})$. On the other hand,
\[ \delta h = \delta {} K {} \varepsilon g - \delta {}\varepsilon u= \ell (g) -0=\ell(g),\]
which shows that $\ell (g)=\delta(h) \in\check{B}^1\left(\Uu, Z^{0,0}_{L^2}\right)$. It follows that   \[ B^{0,1}(\Omega)
\subset \ell^{-1}\left(\check{B}^1\left(\Uu,Z^{0,0}_{L^2}\right)\right).\]

For the opposite inclusion, let $g \in  \ell^{-1}\left(\check{B}^1\left(\Uu,Z^{0,0}_{L^2}\right)\right)$, so that
$\ell(g) \in \check{B}^1\left(\Uu,Z^{0,0}_{L^2}\right)$. Then there exists  $h\in \check{C}^0(\Uu, Z^{0,0}_{L^2})$ such that $\ell(g)=\delta h$. We define $u_0\in\check{C}^0(\Uu,A^{0,0}_{L^2}) $ by
\[ u_0= K{} \varepsilon g -h.\]
In fact, $\dbar u_0 = \dbar {} K {} \varepsilon g -\dbar h= \varepsilon g $.
Also,
\[ \delta u_0 =  \delta{} K {}\varepsilon g -\delta h =\ell(g)-\delta h=0.\]
It follows that there is a $u\in  A^{0,0}(\Omega)$ such that $u_0= \varepsilon u$.  Since $\dbar u_0= \varepsilon (g) $, it follows
that $g=\dbar u$, so that $g\in B^{0,1}(\Omega)$. Equation \eqref{eq-ellproperty} is thus established. It follows from   \eqref{eq-ellproperty} that  there is a continuous linear injective map
\begin{equation}\label{eq-ellbar2}
  \ol{\ell}:\frac{\varepsilon^{-1}\left(\check{C}^0(\Uu,B^{0,1}_{L^2})\right)}{ B^{0,1}_{L^2}(\Omega)}\to \check{H}^1(\Uu, Z^{0,0}_{L^2}),
 \end{equation}
 induced by the map \eqref{eq-ell}.

To prove part (1) of the proposition, suppose that $\check{H}^1(\Uu,Z^{0,0}_{L^2})$ is Hausdorff, i.e.,
the subspace $\check{B}^1\left(\Uu, Z^{0,0}_{L^2}\right)$ is closed in $\check{Z}^1\left(\Uu, Z^{0,0}_{L^2}\right)$.
Then by  \eqref{eq-ellproperty}, since $\ell$ is continuous, $B^{0,1}(\Omega)$ is a closed subspace of $Z^{0,1}(\Omega)$, which means that
$H^{0,1}_{L^2}(\Omega)$ is Hausdorff. This completes the proof of part (1).

 To prove part (2), from the hypothesis that $\check{H}^1(\Uu,Z^{0,0}_{L^2})=0$ and the injectivity of the map
 \eqref{eq-ellbar2}, we see that
 \begin{equation}\label{eq-epsiloninverse}
 B^{0,1}_{L^2}(\Omega)= \varepsilon^{-1}\left(\check{C}^0(\Uu,B^{0,1}_{L^2})\right).
 \end{equation}
Consider now the sequence of Hilbert spaces and continuous linear maps:
\begin{equation}\label{eq-exact}
  Z^{0,1}_{L^2}(\Omega)\xrightarrow{\varepsilon} \check{C}^{0}(\Uu, Z^{0,1}_{L^2})\xrightarrow{\delta} \check{C}^1(\Uu, Z^{0,1}_{L^2}),
\end{equation}
which is clearly exact. Therefore, going modulo $ B^{0,1}_{L^2}(\Omega)$, and using \eqref{eq-epsiloninverse}, we obtain an injective continuous map
\[ \ol{\varepsilon}: H^{0,1}_{L^2}(\Omega) \to \frac{\check{C}^{0}(\Uu, Z^{0,1}_{L^2})}{\check{C}^{0}(\Uu, B^{0,1}_{L^2})} \cong \bigoplus_{j=1}^N H^{0,1}_{L^2}(\Omega_j),\]
whose image (by the exactness of  \eqref{eq-exact})  is the subspace of ${\check{C}^{0}(\Uu, Z^{0,1}_{L^2})}/{\check{C}^{0}(\Uu, B^{0,1}_{L^2})} $ given by
\begin{align*}
\img \ol{\varepsilon}&=\frac{\ker\left(\delta:\check{C}^{0}(\Uu, Z^{0,1}_{L^2})\to \check{C}^1(\Uu, Z^{0,1}_{L^2})\right)}{\check{C}^{0}(\Uu, B^{0,1}_{L^2})}\\
&= \left\{(g_k)_{k=1}^N \in \bigoplus_{k=1}^N Z^{0,1}_{L^2}(\Omega_k)\left|\text{for $i\not=j$, } g_j|_{\Omega_{ij}}=g_i|_{\Omega_{ij}}\right.\right\}\bigg/\left(\bigoplus_{k=1}^N B^{0,1}_{L^2}(\Omega_k)\right)\\
&\subset \left\{ \left(\gamma_k\right)_{k=1}^N, \gamma_k \in H^{0,1}_{L^2}(\Omega_k) \left| \text{  for $i\not=j$}, \gamma_i|_{\Omega_{ij}}= \gamma_j|_{\Omega_{ij}}\right.\right\},
\end{align*}
which completes the proof of part (2) of the proposition.

For part (3), note that under the hypothesis $H^{0,1}_{L^2}(\Omega_j)=0$ for each $j$, the space $\bigoplus_{k=1}^N H^{0,1}_{L^2}(\Omega_k)$ vanishes. But by part (2), there is an injective mapping of $H^{0,1}_{L^2}(\Omega)$ into this space, so the conclusion follows.\end{proof}
\section{Case of a hole which is an intersection}
 For an open set $D$ in $\cx^n$, we from now on for convenience of notation denote the space $Z^{0,0}_{L^2}(D)$ by $\mathcal{O}_{L^2}(D)$. Then
$\mathcal{O}_{L^2}(D)= \mathcal{O}(D)\cap L^2(D)$ is  the {\em Bergman space} of $D$,
the space of all holomorphic functions on $D$ which are square integrable with respect to  Lebesgue measure. 
\subsection{Proof of Theorem~\ref{thm-new}}Set $\Omega_j =\wt{\Omega}\setminus K_j$. Then $\Uu=\{\Omega_j, 1\leq j \leq N\}$ is an open cover 
of $\Omega$, and each $\Omega_j$ is connected, and  $H^{0,1}_{L^2}(\Omega_j)$ is Hausdorff for each $j$ by hypothesis. 

We claim that $\check{H}^1(\Uu,\mathcal{O}_{L^2})=0$.  Let $f \in \check{Z}^1(\Uu,\mathcal{O}_{L^2})$, so that $f=(f_{ij})$, where 
$f_{ij}\in \mathcal{O}_{L^2}(\Omega_{ij})$, $1\leq i,j \leq N$.  By hypothesis (3), the open set $\Omega_{ij}=\wt{\Omega}\setminus (K_i\cup K_j)$ is connected. By Hartogs'  phenomenon, each $f_{ij}$ extends to a $\wt{f}_{ij}\in \mathcal{O}_{L^2}(\wt{\Omega})$.  Now since  $\delta f=0$, we have for $1 \leq i,j,k \leq N$ the following equation on $\Omega_{ijk}$:
\[ f_{ij}- f_{jk}+f_{ki}=0.\]
By analytic continuation we have on the whole of $\wt{\Omega}$:
\begin{equation} \label{eq-ancont} \wt{f}_{ij}- \wt{f}_{jk}+\wt{f}_{ki}=0.
\end{equation}
Define a $u\in \check{C}^0(\Uu, \mathcal{O}_{L^2})$ by setting
$ u_1=0 $ on $\Omega_1$, and for $j\geq 2$, $u_j= \wt{f}_{1j}|_{\Omega_j}$. Then on $\Omega_{ij}$ we have
\[ (\delta u)_{ij} =u_j-u_i = \wt{f}_{1j}- \wt{f}_{1i}= \wt{f}_{ij}|_{\Omega_{ij}}= f_{ij}.\]
Therefore $\delta u =f$, i.e., $\check{H}^1(\Uu, \mathcal{O}_{L^2})=0$. It now follows  from Theorem~\ref{thm-leray} that $H^{0,1}_{L^2}(\Omega)$ is Hausdorff.
This completes the proof of Theorem~\ref{thm-new}.

\subsection{Proof of Corollary~\ref{cor-new}} We note that the following stronger form of Corollary~\ref{cor-new} holds, where the Lipschitz domain $V$ can be replaced by a compact set with minimal boundary regularity hypothesis:
\begin{cor}\label{cor-newbis} Let $K\Subset \cx^n$ be a regular compact set  such that $\cx^n\setminus K$ is connected. Suppose that  $K=\bigcap_{j=1}^N \ol{V_j}$, where for $1\leq j \leq N$, $V_j\Subset \cx^n$ is a Lipschitz domain such that $H^{0,n-1}_{W^1}(V_j)=0$. If $\cx^n\setminus V_j$ is connected for each $j$, and $\cx^n\setminus (V_i\cup V_j)$ is connected for each $1\leq i,j \leq N$,  then
$H^{0,n-1}_{W^1}(K)=0$.
\end{cor}

\begin{proof} Recall that here regularity of the compact $K$ is in the sense of \eqref{eq-regularity}, and $H^{0,n-1}_{W^1}(K)=0$ is in the sense explained after \eqref{eq-hw1} above.
   Apply Theorem~\ref{thm-new}, taking $\wt{\Omega}$ to be a ball 
of sufficiently large radius to contain the closure of all the $V_j$'s.  If we set $K_j=\ol{V}_j$, it is easy to verify that all the hypotheses of Theorem~\ref{thm-new} 
hold, so that $H^{0,1}_{L^2}(\Omega)$ is Hausdorff.  The conclusion now follows from 
 Proposition~\ref{prop-hole1}.\end{proof}
\subsection{Proof of Corollary~\ref{cor-new3}} For each $\Omega_j\in \Uu$, by hypothesis we have $H^{0,1}_{L^2}(\Omega_j)=0$, and as we saw in the proof of Theorem~\ref{thm-new} above, $\check{H}^1(\Uu,\mathcal{O}_{L^2})=0$. Consequently, by part (3) of Theorem~\ref{thm-leray} we have $H^{0,1}_{L^2}(\Omega)=0$.

\section{$L^2$-\v{C}ech cohomology of an annulus between product domains}\label{sec-l2cechcomputation}
Let $N\geq 2$, and for  each  $j\in \{1,\dots, N\}$ let $U_j\subset\cb^{n_j}$ be a bounded domain, and $K_j$ 
a  compact set in $\cx^{n_j}$ such that
 $K_j\subset U_j$, and such that the  annulus  $ R_j = U_j\setminus K_j$ is  connected. Let $U=U_1\times\dots\times U_N$ and $K=K_1\times\dots\times K_N$. In this section we consider the domain
 \[ \wb = U\setminus K,\]
which is an annulus between a product domain and a product hole.  We let
  \begin{equation}
 \label{eq-omegaj}
 \Omega_j= U_1\times\dots\times R_j \times \dots\times U_N,\end{equation}
 where the $j$-th factor is $R_j$ and for $k\not=j$, the $k$-th factor is   $U_k$. We call the domains $R_j$ ($j=1,\dots, N$)  the {\em factor annuli} of $\wb$.
 We denote the  collection $\{\Omega_j, 1\leq j \leq N\}$ by $\Uu$, and note that $ \bigcup_{j=1}^N \Omega_j=\wb$, i.e., $\Uu$ is a cover of $\wb$ by open sets.
 In this section, we prove the following:
 \begin{prop}\label{prop-cechcomputation} With $\Uu$ as above, $\check{H}^1(\Uu, \mathcal{O}_{L^2})$ is Hausdorff, and vanishes if $n\geq 3$. 


  \end{prop}
   Note that pseudoconvexity does not play any direct role in the statement of this result. The proof will be based on a closed range property of the restriction map on Bergman spaces (see Lemma~\ref{lem-cr1}). Also, it is true that for 
   $n=2$, the group  $\check{H}^1(\Uu, \mathcal{O}_{L^2})$  is infinite dimensional, though we neither prove nor use this fact.
    A technique  similar to that used in the proof of Proposition~\ref{prop-cechcomputation} 
   was used to compute  the usual Dolbeault cohomology of some non-pseudoconvex domains in \cite{chak-archiv}.

\subsection{Closed range for restriction maps on Bergman spaces}\label{sec-bergman} For  open subsets $U,R\subset\cx^n$, where $R\subset U$, we denote by $\mathcal{O}_{L^2}(U)|_R$ the subspace of
$\mathcal{O}_{L^2}(R)$ consisting of restrictions of functions in $\mathcal{O}_{L^2}(U)$.  

\begin{lem}\label{lem-cr1}Let $U$ be a bounded domain in $\cx^n$, and let $K\subset U$ be a  compact subset. Set $R=U\setminus K$ and when $n\geq 2$, assume that  $R$ is connected. Then  $\mathcal{O}_{L^2}(U)|_R$ is a closed subspace of the Hilbert space $\mathcal{O}_{L^2}(R)$. 
\end{lem}
\begin{proof}If $n\geq 2$,  then by the  Hartogs extension theorem,  $\mathcal{O}_{L^2}(U)|_R=  \mathcal{O}_{L^2}(R)$, so the assertion is 
correct.  For $n=1$, by a classical argument (cf.  \cite[p. 143 ]{Ahl} or \cite[p. 195, Proposition~1.1]{conway}) there is a neighborhood $W$ of $K$ such that $W$ is contained in $U$, the boundary $bW$ is a union of finitely many closed polygons, and for any holomorphic function in $U$, we have a
Cauchy representation
\[ f(z)= \frac{1}{2\pi i}\int_{bW}\frac{f(\zeta)}{\zeta-z}d\zeta\]
valid for each $z\in K$.

Let $f_\nu \in \mathcal{O}_{L^2}(U)$, and suppose that $f_\nu|_R\to g$ in $\mathcal{O}_{L^2}(R)$ as  $\nu\to \infty$. By Bergman's inequality (cf. \cite[p. 155]{remmert}),  $\{f_\nu\}$ converges uniformly to $g$ when restricted to the compact subset  $bW$ of $R$.  Representing the holomorphic function $f_\nu$ on $K$ as the Cauchy integral over $bW$  as in the above formula, we see that $\{f_\nu\}$ converges uniformly on $K$ to a function given by the Cauchy integral of $g$.  It follows that $g$ extends to a function in $\mathcal{O}_{L^2}(U)$, which proves the lemma.
\end{proof}
\medskip

Let $\mathcal{Q}\csor\mathcal{R}$ denote the Hilbert  tensor product of Hilbert spaces $\mathcal{Q}$ and $\mathcal{R}$. In our application, we only consider Hilbert tensor products where for some domain $D\subset\cx^m$, the space
$\mathcal{Q}$ is a closed subspace of $\mathcal{O}_{L^2}(D)$ , and for some domain $V\subset \cx^n$, the space $\mathcal{R}$ is a closed subspace of $\mathcal{O}_{L^2}(V)$. Then $\mathcal{Q}\csor\mathcal{R}$ is the closure in
$\mathcal{O}_{L^2}(D\times V)$ of the linear span of functions of the form $(f\tensor g)(z,w)=f(z)g(w)$, where $f\in \mathcal{Q}$ and $g\in \mathcal{R}$. See \cite{prod} for more details on Hilbert tensor products.


 \subsection{Proof of Proposition~\ref{prop-cechcomputation}}
   We first consider the case when the number of factors $N=2$. To conclude that $\check{H}^1(\Uu, \bs)$ is Hausdorff, we need to show that the coboundary map
 \[ \delta: \check{C}^0(\Uu,\bs)\to  \check{C}^1(\Uu,\bs)\]
 has closed range. As a topological vector space, $\check{C}^1(\Uu,\bs)$ is  simply $\bs(\Omega_{12})$, since there is only one double intersection. Therefore, the closed range of
 $\delta$ will follow, if we show that   the map
 \[ \mathcal{O}_{L^2}(\Omega_1)\oplus \mathcal{O}_{L^2}(\Omega_2)\to \mathcal{O}_{L^2}(\Omega_{12}),\]
 given by
 \[ \delta(h_1, h_2)= h_2|_{\Omega_{12}} - h_1|_{\Omega_{12}}\]
 has closed range. Recall that $\Omega_1=R_1\times U_1$, $\Omega_2=U_1\times R_2$ and $\Omega_{12}=R_1\times R_2$.


Since  for $j=1,2$, by Lemma~\ref{lem-cr1} $\oc_{L^2}({U}_j)|_{{R}_j}$ is closed in $\oc_{L^2}({R}_j)$ we obtain a direct sum decomposition:
\[ \oc_{L^2}({R}_j)= \oc_{L^2}({U}_j)|_{{R}_j}\oplus \left(\mathcal{O}_{L^2}({U}_j)|_{{R}_j}\right)^\perp,\]
and it follows by the distributivity of the Hilbert tensor product over direct sums that
\begin{align*}\bs(\Omega_{12})&=
 \mathcal{O}_{L^2}({R}_1\times {R}_2)\\&= \mathcal{O}_{L^2}({R}_1)\csor  \mathcal{O}_{L^2}({R}_2)\\
 &= \bigoplus_{j=1}^4 E_j,
\end{align*}
where
\begin{align*}
 E_1&= \mathcal{O}_{L^2}({U}_1)|_{{R}_1}\csor
\mathcal{O}_{L^2}({U}_2)|_{{R}_2},\\
E_2 &= \left(\mathcal{O}_{L^2}({U}_1)|_{{R}_1}\right)^\perp\csor
\mathcal{O}_{L^2}({U}_2)|_{{R}_2},\\
E_3&= \mathcal{O}_{L^2}({U}_1)|_{{R}_1}\csor
\left(\mathcal{O}_{L^2}({U}_2)|_{{R}_2}\right)^\perp, &\text{  and }\\
E_4 &= \left(\mathcal{O}_{L^2}({U}_1)|_{{R}_1}\right)^\perp\csor
\left(\mathcal{O}_{L^2}({U}_2)|_{{R}_2}\right)^\perp,
\end{align*}
and for $j=1,2$,
$\left(\mathcal{O}_{L^2}(U_j)|_{R_j}\right)^\perp$
denotes the orthogonal complement of the closed subspace $\mathcal{O}_{L^2}(U_j)|_{R_j}$ in $\mathcal{O}_{L^2}(R_j)$.
Note that
\[ E_1 \oplus E_3= \mathcal{O}_{L^2}({U}_1)|_{{R}_1}\csor \oc_{L^2}({R}_2)=\mathcal{O}_{L^2}({U}_1\times {R}_2)|_{{R}_1\times {R}_2}\subset \img \delta\]
and
\[ E_1 \oplus E_2= \oc_{L^2}({R}_1)\csor \mathcal{O}_{L^2}({U}_2)|_{{R}_2} =\mathcal{O}_{L^2}({R}_1\times {U}_2)|_{{R}_1\times {R}_2}\subset \img \delta,\]
and by definition of $\delta$,
\begin{align*}
\img \delta &= \mathcal{O}_{L^2}({U}_1\times {R}_2)|_{{R}_1\times {R}_2}+ \mathcal{O}_{L^2}({R}_1\times {U}_2)|_{{R}_1\times {R}_2}\\
&= (E_1 \oplus E_3)+ ( E_1 \oplus E_2)\\
&=  E_1 \oplus E_2 \oplus E_3\\
&=\left(E_4\right)^\perp\\
&= \left(\left(\mathcal{O}_{L^2}(U_1)|_{R_1}\right)^\perp\csor
\left(\mathcal{O}_{L^2}(U_2)|_{R_2}\right)^\perp\right)^{\perp},
\end{align*}
Where the outer $\perp$ in the last line, and the $\perp$ in the previous to last line denote orthogonal complementation in
$\bs(\Omega_{12})$. It follows that $\img \delta$ is closed,
and we obtain an isomorphism of Hilbert spaces
\[ \check{H}^1(\Uu,\bs)\cong E_4= \left(\mathcal{O}_{L^2}({U}_1)|_{{R}_1}\right)^\perp\csor
\left(\mathcal{O}_{L^2}({U}_2)|_{{R}_2}\right)^\perp,\]
valid when the number of factors $N=2$ and the dimension $n\geq 2$.

Now assume $n\geq 3$. Since $n=n_1+n_2$,  there is a $j\in \{1,2\}$ such that $n_j\geq 2$.  By Hartogs' phenomenon, $ \mathcal{O}_{L^2}(U_j)|_{R_j} = \mathcal{O}_{L^2}(R_j)$ so we have $ (\mathcal{O}_{L^2}(U_j)|_{R_j})^\perp=\{0\}$, and consequently,
$\check{H}^1(\Uu,\mathcal{O}_{L^2})=0$.


 Now we will show that for $N\geq 3$ (which forces $n\geq 3$), we again have $\check{H}^1(\Uu,\mathcal{O}_{L^2})=0$.  A similar  result (with one dimensional factors) was  proved by   Frenkel  \cite[Proposition~31.1]{frenkel} for the structure sheaf $\oc$.

Let $N\geq 3$, and suppose that $f=(f_{ij})_{i<j}\in \check{Z}^1(\Uu,\mathcal{O}_{L^2})$. We will show that there is a $u\in \check{C}^0(\Uu, \mathcal{O}_{L^2})$ such that $f=\delta u$.

Denote by $\mathcal{E}$ the subspace of  $\check{Z}^1(\Uu,\mathcal{O}_{L^2})$ consisting of $f$ such that each $f_{ij}\in\bs(\Omega_{ij})$ extends holomorphically to a function in $\bs(U)$, so that on $\Omega_{ij}$:
\begin{equation}\label{eq-fij1}
f_{ij} \in \bs(U_i)|_{R_i}\csor \bs(U_j)|_{R_j}\csor \bs(U_{ij}') \,
\end{equation}
where the tensor product has been reordered (and we will do this in the sequel without further comment)
 and $U_{ij}'$ is the product of all the $U_k$'s except $U_i$ and $U_j$. Define $u\in \check{C}^0(\Uu, \bs)$ by setting $u_1=0$ on $\Omega_1$, and for $j\geq 2$, $u_j=f_{1j}|_{\Omega_j}$, where we continue to denote the extension of $f_{ij}$ to $U$ by the same symbol. Then on the set $\Omega_{ij}$ we have
 \[ u_j-u_i=f_{1j}-f_{1i}=f_{ij},\]
 since $\delta f=0$. We have  therefore $\delta u=f$.

Hence we may assume without loss of generality that $f\in \check{Z}^1(\Uu,\mathcal{O}_{L^2})$ actually belongs to the orthogonal complement
$\mathcal{E}^\perp$ of $\mathcal{E}$ in $\check{Z}(\Uu,\mathcal{O}_{L^2})$.
Noting that  $\bs(\Omega_{ij})=\bs(R_i) \csor \bs(R_j)\csor \bs(U_{ij}')$,  and that for each $k$ we have  $\bs(R_k)= \bs(U_k)|_{R_k}\oplus\bs(U_k)|_{R_k}^\perp$ by Lemma~\ref{lem-cr1}, we have from \eqref{eq-fij1} for each $i<j$,
\begin{equation}\label{eq-fij2}
f_{ij} \in S_1\bigoplus S_2\bigoplus S_3 \subset \bs(\Omega_{ij}),
\end{equation}
where
\begin{align*}
S_1&= \bs(U_i)|_{R_i}^\perp \csor \bs(U_j)|_{R_j}\csor \bs(U_{ij}'),\\
S_2&= \bs(U_i)|_{R_i}\csor \bs(U_j)|_{R_j}^\perp\csor \bs(U_{ij}'),\\
S_3&= \bs(U_i)|_{R_i}^\perp\csor \bs(U_j)|_{R_j}^\perp\csor \bs(U_{ij}').
\end{align*}
We claim that {\em the component of $f_{ij}$ along $S_3$ vanishes} i.e., we have
\begin{equation}\label{eq-fij3}
f_{ij}\in S_1\bigoplus S_2.
\end{equation}

To prove the claim,  since $N\geq 3$, there is a $k\in \{1,2, \dots, N\}$ such that $k$ is distinct from both $i$ and $j$.
  Let  $U_{ijk}'$ be the
product of all the $U_\ell$'s except $U_i, U_j$ and $U_k$. Consider the restriction map $\rho$
 from $\bs(\Omega_{ij})$ to
$\bs(\Omega_{ijk})$.  Since $\Omega_{ij}= R_i\times R_j\times R_k\times U_{ijk}'$ and
$\Omega_{ijk}=  R_i\times R_j\times R_k\times R_{ijk}'$, the restriction map is a tensor product:
\[ \rho = {\rm id}_{\bs(R_i)}\csor  {\rm id}_{\bs(R_j)}\csor \rho_k\csor 
     {\rm id}_{\bs(U_{ijk}')},\]
     where $\rho_k$ is the restriction map from $\bs(U_k)$ to $\bs(R_k)$.  Consequently, we have a tensor product 
     representation
     \[ S_3|_{\Omega_{ijk}}= \bs(U_i)|_{R_i}^\perp\csor \bs(U_j)|_{R_j}^\perp\csor\bs(U_k)|_{R_k}\csor \bs(U_{ijk}').\]
     Denote by $p_3$ the orthogonal projection from $\bs(\Omega_{ij})$ to $S_3$, and by $P_3$ the orthogonal projection from $\bs(\Omega_{ijk})$ onto
$S_3|_{\Omega_{ijk}}$. Then we have the diagram
   \begin{equation}\label{eq-diag}  
\begin{CD}
\mathcal{O}_{L^2}(\Omega_{ij})     @> p_3>>  S_3\\
@VV\rho V        @VV\rho V\\\mathcal{O}_{L^2}(\Omega_{ijk})     @>P_3>> S_3|_{\Omega_{ijk}}
\end{CD}\end{equation}
  which commutes, since both $\rho\circ p_3$ and $P_3 \circ  \rho$ are equal to $\pi_i\csor \pi_j\csor \rho_k \csor  {\rm id}_{\bs(U_{ijk}')}$, where $\pi_i: \bs(U_i)\to \bs(U_i)|_{R_i}^\perp$ and $\pi_j: \bs(U_j)\to \bs(U_j)|_{R_j}^\perp$ are the orthogonal projections.
   
 We  show that $P_3(f_{ij}|_{\Omega_{ijk}})=0$.  From $\delta f=0$, we conclude that
\begin{equation}\label{eq-cocycle}
 f_{ij}|_{\Omega_{ijk}} =f_{ik}|_{\Omega_{ijk}}-f_{jk}|_{\Omega_{ijk}}.
\end{equation}
Now,
\[f_{ik}|_{\Omega_{ijk}}\in  \bs(U_i)|_{R_i}^\perp\csor \bs(U_j)|_{R_j}\csor \bs(U_k)|_{R_k}^\perp\csor \bs(U_{ijk}'),\]
and
\[ f_{jk}|_{\Omega_{ijk}}\in \bs(U_i)|_{R_i}\csor \bs(U_j)|_{R_j}^\perp\csor \bs(U_k)|_{R_k}^\perp\csor \bs(U_{ijk}'),\] so that
both $f_{ik}|_{\Omega_{ijk}}$ and $f_{jk}|_{\Omega_{ijk}}$ lie in subspaces of $\bs(\Omega_{ijk})$ which are orthogonal to $S_3|_{\Omega_{ijk}}$.
Therefore $P_3(f_{ik}|_{\Omega_{ijk}})= P_3(f_{jk}|_{\Omega_{ijk}})=0$. Therefore, by \eqref{eq-cocycle}, 
we see that  $P_3(f_{ij}|_{\Omega_{ijk}})=0$. 

Now, by the commutativity of the diagram \eqref{eq-diag}, we have that 
\[ \rho(p_3(f_{ij}))= P_3(\rho(f_{ij}))= P_3(f_{ij}|_{\Omega_{ijk}})=0,\]
and since by analytic continuation $\rho$ is injective, we have $p_3(f_{ij})=0$. Therefore,  the claim \eqref{eq-fij3} follows.

Denote now by $P_1$ and $P_2$ the projections from $\bs(\Omega_{ijk})$ onto $S_1|_{\Omega_{ijk}}$ and $S_2|_{\Omega_{ijk}}$ respectively. Note that 
\begin{equation}\label{eq-s1res}
\begin{aligned}
S_1|_{\Omega_{ijk}}&= \bs(U_i)|_{R_i}^\perp \csor \bs(U_j)|_{R_j}\csor\bs(U_k)|_{R_k} \csor\bs(U_{ij}'),\\
S_2|_{\Omega_{ijk}}&= \bs(U_i)|_{R_i}\csor \bs(U_j)|_{R_j}^\perp\csor\bs(U_k)|_{R_k} \csor\bs(U_{ij}'),\\
\end{aligned}
\end{equation}
By the representation in \eqref{eq-s1res} above, we see that  $P_1(f_{ij})$ extends holomorphically to an element $(-u^{ij}_i)$ of $\bs(\Omega_i)$, and $P_2(f_{ij})$ extends holomorphically to an element $u_j^{ij}$ of $\bs(\Omega_j)$.  Therefore we have \begin{equation}\label{eq-fij4}
 f_{ij} = u^{ij}_j|_{\Omega_{ij}} - u^{ij}_i|_{\Omega_{ij}}.
\end{equation}
We note two features of this decomposition. First, it is independent of the choice of the auxiliary index $k$, since the decomposition \eqref{eq-fij2} in no way depends on $k$. Second, by construction, $u^{ij}_i|_{\Omega_{ijk}}$ and $u^{ij}_j|_{\Omega_{ijk}}$ belong to the orthogonal subspaces
$S_1|_{\Omega_{ijk}}$ and $S_2|_{\Omega_{ijk}}$ of $\bs(\Omega_{ijk})$. (The orthogonality is immediate from 
\eqref{eq-s1res} above).

We claim that if $k$ is an index distinct from $i$ and $j$, we have $u^{ij}_i=u^{ik}_i$. To see this, we substitute expressions like \eqref{eq-fij4} into \eqref{eq-cocycle}, and obtain, that when restricted to $\Omega_{ijk}$, we have
\[ (u^{jk}_k - u^{jk}_j)- (u^{ik}_k- u^{ik}_i)+ (u^{ij}_j- u^{ij}_i)=0,\]
which gives rise to the condition that
\[ (u^{jk}_k|_{\Omega_{ijk}}- u^{ik}_k|_{\Omega_{ijk}})
+ (u^{ij}_j|_{\Omega_{ijk}}- u^{jk}_j|_{\Omega_{ijk}}) + (u^{ik}_i|_{\Omega_{ijk}} - u^{ij}_i|_{\Omega_{ijk}})=0.\]
The three terms of the above sum belong to three orthogonal subspaces of $\bs(\Omega_{ijk})$.
The first term is in $\bs(U_i)|_{R_i}\csor \bs(U_j)|_{R_j}\csor\bs(U_k)|_{R_k}^\perp \csor\bs(U_{ij}')$,
the second term is in $S_2|_{\Omega_{ijk}}$ and the third term is in $S_1|_{\Omega_{ijk}}$, where the notation is as in
\eqref{eq-s1res}. Therefore, all three terms vanish.  By analytic continuation,
 it follows that
there is for each $i\in \{1,\dots, N\}$ an $u_i\in \bs(\Omega_i)$ such that if $u=(u_i)_{i=1}^N \in \check{C}^0(\Uu, \bs)$, then
$\delta u =f$. It follows now that $\check{H}^1(\Uu,\bs)=0$ if $N\geq 3$.




\section{Proof of Theorem~\ref{thm-W1}}
 \subsection{The $L^2$-Dolbeault cohomology of an annulus  between product domains}
 Let $\wb=U\setminus K\subset\cx^n$ be as in Section~\ref{sec-l2cechcomputation},
 an annulus between the products $U=U_1\times\dots\times U_N$ and $K=K_1\times\dots\times K_N$,
 and let $R_j=U_j\setminus K_j\subset \cx^{n_j}$ denote the factor annuli for $j=1,\dots, N$. In this section, we compute the $L^2$-Dolbeault cohomology of $\wb$:

\begin{prop}\label{prop-specialcase}Suppose that for each $j$,  $H^{0,1}_{L^2}(R_j)$ is  Hausdorff. Then  $H^{0,1}_{L^2}(\wb)$ is  Hausdorff. Further, if $H^{0,1}_{L^2}(U_j)=0$ for each $j$, then  $H^{0,1}_{L^2}(\wb)$ vanishes if $n\geq 3$.
\end{prop}
\begin{proof}
By Proposition~\ref{prop-cechcomputation}, the \v{C}ech group $\check{H}^1(\Uu,Z^{0,0}_{L^2})$ is
Hausdorff. Now in the cover $\Uu$, we can write $\Omega_j=U_1\times \dots\times R_j \times \dots U_N$. If $n_j=1$, then $H^{0,1}_{L^2}(U_j)=0$ (and therefore Hausdorff), and if $n_j\geq 2$, then by Proposition~\ref{prop-hole2}, since each $H^{0,1}_{L^2}(R_j)$ is  Hausdorff, and $R_j=U_j\setminus K_j$ is an annulus,   it follows that each $H^{0,1}_{L^2}(U_j)$ is  also Hausdorff.
It follows from the results  of \cite{prod, spec}  regarding  the $L^2$-cohomology of product domains, that since $\Omega_j$ is a product of domains whose $L^2$-Dolbeault cohomology
is Hausdorff in degrees $(0,0)$ and $(0,1)$, the cohomology $H^{0,1}_{L^2}(\Omega_j)$ is also Hausdorff.
Now by Part (1) of Theorem~\ref{thm-leray}, since we have covering $\Uu$ of $\wb$, such that for each $\Omega_j\in \Uu$ we have $H^{0,1}_{L^2}(\Omega_j)$ Hausdorff, as well as the \v{C}ech group $\check{H}^1(\Uu,Z^{0,0}_{L^2})$
Hausdorff, we conclude that $H^{0,1}_{L^2}(\wb)$ is Hausdorff.

Now let $n\geq 3$.  We apply part (2) of Theorem~\ref{thm-leray}, and assume that $H^{0,1}(U_j)=0$ for each $j$. By Proposition~\ref{prop-cechcomputation},
 $\check{H}^1(\mathfrak{U}, \mathcal{O}_{L^2})=0$, so the cohomology $H^{0,1}_{L^2}(\Omega)$ is isomorphic
 to a subspace of
 $\bigoplus_{k=1}^N H^{0,1}_{L^2}(\Omega_k) $ contained in
\[
\left\{ \left(\gamma_k\right)_{k=1}^N, \gamma_k \in H^{0,1}_{L^2}(\Omega_k) \left| \text{  for $i\not=j$}, \gamma_i|_{\Omega_{ij}}= \gamma_j|_{\Omega_{ij}}\right.\right\}.
\]
It is therefore sufficient to show that the above space vanishes.
Recall that
\[ \Omega_1=  R_1\times U_2 \times U_{12}', \text{ and }  \Omega_2= U_1\times R_2\times  U_{12}',\]
where $U_{12}'= U_3\times\dots \times U_{N}$. Then
\[ \Omega_{12}= R_1\times R_2 \times  U_{12}'.\]
We have  by the K\"unneth formula for $L^2$-cohomology (cf. \cite{prod, spec}),
\[ H^{0,1}_{L^2}(\Omega_1)= H^{0,1}_{L^2}(R_1)\csor H^{0,0}_{L^2}(U_2) \csor H^{0,0}_{L^2}(U_{12}'),\]
and
\[ H^{0,1}_{L^2}(\Omega_2)= H^{0,0}_{L^2}(U_1)\csor H^{0,1}_{L^2}(R_2) \csor H^{0,0}_{L^2}(U_{12}'),\]
where the other terms vanish since $H^{0,1}(U_j)=0$ for each $j$. Similarly we obtain
\begin{equation}
 H^{0,1}(\Omega_{12})
 =  H^{0,1}_{L^2}(R_1) \csor H^{0,0}_{L^2}(R_2)\csor H^{0,0}(U_{12}')\bigoplus  H^{0,0}_{L^2}(R_1) \csor H^{0,1}_{L^2}(R_2)\csor H^{0,0}_{L^2}(U_{12}'),
 \label{eq-omegaijcohomology}
 \end{equation}
 Note also that the two direct summands in \eqref{eq-omegaijcohomology} are orthogonal to each other thanks to the tensor nature of the inner product on $H^{0,1}_{L^2}(\Omega_{12})$ (see \cite{prod}).  Consider now the restriction map
 \begin{equation}\label{eq-restriction1}
   H^{0,1}_{L^2}(\Omega_1)\to H^{0,1}_{L^2}(\Omega_{12}),\end{equation}
written as $\gamma\to \gamma|_{\Omega_{12}}$ whose image is
\begin{equation}\label{eq-range1}
 H^{0,1}_{L^2}(\Omega_1)|_{\Omega_{12}}\cong H^{0,1}_{L^2}(R_1) \csor H^{0,0}(U_2)|_{R_2} \csor H^{0,0}_{L^2}(U_{12}'),
\end{equation}
 and the map \eqref{eq-restriction1} may be represented as a tensor product of maps
 \[ {\rm id}_{H^{0,1}_{L^2}(R_1)} \csor\rho\,\csor {\rm id}_{H^{0,0}_{L^2}(U_{12}')},\]
 where
 \[  \rho: H^{0,0}_{L^2}(U_2)\to H^{0,0}_{L^2}(R_2)\]
 is the restriction map $f\mapsto f|_{R_2}$ from $H^{0,0}_{L^2}(U_2)= \mathcal{O}_{L^2}(U_2)$ to $H^{0,0}_{L^2}(R_2)=\mathcal{O}_{L^2}(R_2)$,
 which is injective by analytic continuation, since $R_2$ is connected. Therefore the map \eqref{eq-restriction1} is also injective, being the tensor
 product of injective maps. A similar reasoning shows  that the restriction map
 \begin{equation}\label{eq-restriction2}
   H^{0,1}_{L^2}(\Omega_2)\to H^{0,1}_{L^2}(\Omega_{12}),\end{equation}
is also injective, and has image
\begin{equation}\label{eq-range2}
H^{0,1}_{L^2}(\Omega_2)|_{\Omega_{12}}\cong H^{0,0}_{L^2}(U_1)|_{R_1} \csor H^{0,1}_{L^2}(R_2) \csor H^{0,0}_{L^2}(U_{12}').
\end{equation}

 The subspaces of  $H^{0,1}_{L^2}(\Omega_{12})$ given by \eqref{eq-range1} and \eqref{eq-range2} are orthogonal, being contained in different summands of the orthogonal direct sum \eqref{eq-omegaijcohomology}, so that we have $ H^{0,1}_{L^2}(\Omega_1)|_{\Omega_{12}} \cap  H^{0,1}_{L^2}(\Omega_2)|_{\Omega_{12}}=\{0\}$, and since this reasoning applies if 1 and 2 are replaced by $i$ and $j$ with $i\not= j$, we see that
 \begin{equation}     H^{0,1}_{L^2}(\Omega_i)|_{\Omega_{ij}} \cap  H^{0,1}_{L^2}(\Omega_j)|_{\Omega_{ij}}=\{0\},
\label{eq-intersection}
\end{equation}
whenever $i,j\in \{1,\dots, N\}$.
Now in \eqref{eq-image2}, let the $N$-tuple $(\gamma_k)_{k=1}^N \in \bigoplus_{k=1}^N H^{0,1}_{L^2}(\Omega_k)$ be such that $\gamma_i|_{\Omega_{ij}}=
\gamma_j|_{\Omega_{ij}}$ whenever $i\not =j$. Therefore, by \eqref{eq-intersection}, we have $\gamma_i|_{\Omega_{ij}}=0$ for all $i\not=j$.
Now the same reasoning that shows that the map \eqref{eq-restriction1} is injective shows that the restriction map  $ H^{0,1}_{L^2}(\Omega_i)\to H^{0,1}(\Omega_{ij})$ is injective when $i\not=j$, and this shows that $\gamma_i=0$ for each $i$.  Part (2) of Theorem~\ref{thm-leray} now shows that
$H^{0,1}_{L^2}(\wb)=0$.
\end{proof}
\subsection{Proof of Theorem~\ref{thm-W1}}
We begin by noting the following variant of Theorem~\ref{thm-W1} with minimal boundary regularity, when the factors are each one dimensional:
\begin{prop} \label{prop-W1}For $j=1,\dots, n$, let $K_j$ be a compact subset of $\cx$, and let $K=K_1\times \dots\times K_n\subset\cx^n$ be their cartesian product. Then if $K$ is regular, we have $H^{0,n-1}_{W^1}(K)=0$.
\end{prop}

\begin{proof}For each $j$, let $U_j$ be a large disc containing the compact $K_j$. Apply Proposition~\ref{prop-hole1}, with $\wt{\Omega}= U=U_1\times \dots\times U_n$, 
and $K=K_1\times\dots\times K_n$. A simple topological reasoning shows that for large enough $U_j$, the annulus $U\setminus K$ is connected.  Since $U$ is pseudoconvex, the result follows.
\end{proof}

The proof of the general case is similar:
\begin{proof}[Proof of  Theorem~\ref{thm-W1}]
For each $j$, choose a large ball $U_j\subset \cx^{n_j}$ such that $V_j\Subset U_j$. If $n_j\geq 2$,  since $H^{0,n_j-1}_{W^1}(V_j)=0$,
and $H^{0,1}_{L^2}(U_j)=0$, it follows by Theorem~\ref{thm-hole} that $H^{0,1}_{L^2}(R_j)$ is Hausdorff, where $R_j=U_j\setminus \ol{V}_j$. In case $n_j=1$, then $H^{0,1}_{L^2}(R_j)=0$, and so is Hausdorff. Now let $\wb=U\setminus \ol{V}$, where $U=U_1\times U_2\times\dots \times U_N$.
Therefore, by Proposition~\ref{prop-specialcase}, we see that $H^{0,1}_{L^2}(\wb)$ is Hausdorff. We now invoke
Theorem~\ref{thm-hole} again to conclude that $H^{0,n-1}_{W^1}(V)=0$.
\end{proof}

\section{Proof of Theorem~\ref{thm-cohomology}}


\subsection{Extension of solvability}
In this section, in order to prove Theorem~\ref{thm-cohomology} we consider the following situation. Let  $D \Subset \wt{\Omega}$ be bounded domains in
$\cx^n$, and let $K$ be a compact set contained in $D$. We consider the relation between the cohomologies of
the annuli
\[ \Omega= \wt{\Omega}\setminus K\]
and
\[ \wb=D\setminus K.\]
\begin{prop}\label{prop-siu} Suppose that   in some degree $(p,q)$, $q\geq 1$,  we have  $H^{p,q}_{L^2}(\wt{\Omega})=0$, and  the $L^2$-cohomology $H^{p,q}_{L^2}(\wb)$ is Hausdorff. Then:

(a) $H^{p,q}_{L^2}(\Omega)$ is Hausdorff.

(b) The natural restriction map on cohomology
\begin{equation}\label{eq-restr}
H^{p,q}_{L^2}(\Omega)\to H^{p,q}_{L^2}(\wb)
\end{equation}
 is injective. Consequently, if  $H^{p,q}(\wb)=0$ then  $H^{p,q}(\Omega)=0$.

\end{prop}
The proof is based on the following observation:
\begin{lem}\label{lem-siu}With hypothesis and notation as above, there is a constant $C>0$ with the following property.
Suppose that $g\in Z^{p,q}_{L^2}(\Omega)$ is such that the restriction $g|_{\wb}$ is in $B^{p,q}_{L^2}(\wb)$. Then
there is a $u\in A^{p,q-1}_{L^2}(\Omega)$ such that $\dbar u=g$ and
\[ \norm{u}_{L^2(\Omega)} \leq C \norm{g}_{L^2(\Omega)}.\]
\end{lem}
\begin{proof}  We denote by $C$ any constant that depends solely on the geometry of the domains $\wb$ and $\wt{\Omega}$, and $C$ may have different values and different occurrences. Since $H^{p,q}_{L^2}(\wb)$ is Hausdorff, there is a $u_0\in A^{p,q-1}_{L^2}(\wb)$ such that $\dbar u_0= g|_{\wb}$, and we have an estimate
\begin{equation}\label{eq-e0}
 \norm{u_0}_{L^2(\wb)} \leq C \norm{g}_{L^2(\wb)}.
\end{equation}
We take a cutoff $\chi\in \mathcal{C}^\infty_0(D)$ such that $\chi\equiv 1$ near $K$. As usual, we assume that after multiplying by a cutoff, we extend functions and forms by zero outside the support of the cutoff. We note that $\dbar (\chi u_0)= \dbar \chi \wedge u_0 + \chi \cdot g$ on $\Omega$,
so that we have an estimate
\begin{equation}\label{eq-e1}
 \norm{\dbar (\chi u_0)}_{L^2(\Omega)} \leq C\norm{g}_{L^2(\Omega)}.
\end{equation}
Let $h= g- \dbar (\chi u_0).$ Since by hypothesis $g\in Z^{p,q}_{L^2}(\Omega)$, we see that  $h\in Z^{p,q}_{L^2}(\Omega)$, and 
near $K$, we have $h= g-\dbar u_0 =0$. If we define
\[ h^\sharp = \begin{cases} h &\text{on $\Omega$}\\ 0 &\text{ on $K$},\end{cases}\]
then $h^\sharp$ belongs to $Z^{p,q}_{L^2}(\wt{\Omega})$, and we have
\begin{equation}\label{eq-e2}
 \norm{h^\sharp}_{L^2(\wt{\Omega})} = \norm{h}_{L^2(\Omega)}\leq C\norm{g}_{L^2(\Omega)},
\end{equation}
where the last estimate follows immediately from the definition of $h$ and \eqref{eq-e1}.  Since $H^{p,q}_{L^2}(\wt{\Omega})=0$,  by the open mapping theorem
there is a $v\in A^{p,q-1}_{L^2}(\wt{\Omega})$ such that
$\dbar v=h^\sharp$ and
\[ \norm{v}_{L^2(\wt{\Omega})} \leq {C} \norm{h^\sharp}_{L^2(\wt{\Omega})}={C} \norm{h}_{L^2({\Omega})} ,\]
where  the last equality holds since $h^\sharp=0$ on $K$. So, if we set
\[ u= v|_{\Omega}+ \chi u_0,\]
then
\[ \dbar u = \dbar v |_{\Omega}+ \dbar (\chi u_0)= h^\sharp|_{\Omega} +\dbar (\chi u_0)=h+\dbar (\chi u_0)=g,\]
by the definition of $h$. Further we have
\begin{align*}
  \norm{u}_{L^2(\Omega)}&\leq \norm{v}_{L^2(\Omega)} + \norm{\chi u_0}_{L^2(\Omega)}\\
  &\leq {C} \norm{h}_{L^2({\Omega})}+ C  \norm{u_0}_{L^2(\wb)}\\
  & \leq C \norm{g}_{L^2(\Omega)},
  \end{align*}
  using \eqref{eq-e0} and \eqref{eq-e2}.
\end{proof}

\begin{proof}[Proof of Proposition~\ref{prop-siu} ](a) Let $\lambda:Z^{p,q}_{L^2}(\Omega)\to Z^{p,q}_{L^2}(\wb)$ be the restriction map $g\mapsto g|_{\wb}$. It follows immediately from   Lemma \ref{lem-siu}
 that $B^{p,q}_{L^2}(\Omega)= \lambda^{-1}(B^{p,q}_{L^2}(\wb))$. But since $B^{p,q}_{L^2}(\wb)$ is closed in $Z^{p,q}_{L^2}(\wb)$ it follows
that $B^{p,q}_{L^2}(\Omega)$ is closed in $Z^{p,q}_{L^2}(\Omega)$.

(b) Let $g\in Z^{p,q}_{L^2}(\Omega)$ be such that  ${\rm class}(g)\in H^{p,q}_{L^2}(\Omega)$ is in the kernel of
the restriction map. Then $g|_{\wb}$ is in $B^{p,q}_{L^2}(\wb)$, so that by Lemma~\ref{lem-siu} above, $g\in B^{p,q}_{L^2}(\Omega)$, and $g$ represents 0 in $H^{p,q}_{L^2}(\Omega)$.
\end{proof}

\subsection{Proof of Theorem~\ref{thm-cohomology}}
We let $\wb= U\setminus K$ in Proposition~\ref{prop-siu}. Let  $R_j=U_j\setminus K_j$ be the factor annuli. Therefore $H^{0,1}_{L^2}(R_j)$ is Hausdorff for each $j$, and 
consequently by Proposition~\ref{prop-specialcase}, we know that
$H^{0,1}_{L^2}(\wb)$ is Hausdorff and vanishes if $n\geq 3$. Since by hypothesis, $H^{0,1}_{L^2}(\wt{\Omega})=0$,
we conclude by Part (a) of Proposition~\ref{prop-siu} that $H^{0,1}_{L^2}(\Omega)$ is Hausdorff and by part (b) that
$H^{0,1}_{L^2}(\Omega)=0$ if $n\geq 3$.


\end{document}